\documentclass[pdflatex,sn-mathphys-num]{sn-jnl}
\usepackage{rotating}
\usepackage{subfig}
\usepackage{amsmath}
\usepackage[english]{babel}
\usepackage{setspace}
\usepackage{amsfonts}
\usepackage{graphicx}
\usepackage{epstopdf}
\usepackage{algorithm}
\usepackage{algorithmic}
\usepackage{mathtools}
\usepackage{multirow}
\usepackage{booktabs} 
\usepackage{makecell}
\usepackage{url}
\usepackage{cleveref}
\usepackage{amssymb}
\usepackage{cases}
\usepackage{mathdots}

\usepackage{tikz}
\usetikzlibrary{calc,arrows.meta}

\newcommand{\HPanelAWidth}{0.32815\linewidth}
\newcommand{\HPanelBWidth}{0.49185\linewidth}
\newcommand{\HPanelGapWidth}{0.120\linewidth}

\newcommand{\HViewDistance}{30}
\newcommand{\HVanishingX}{48}
\newcommand{\HVanishingY}{-18}

\def\Hreq{black!34}
\def\Hfront{black!34}
\def\Hside{black!34}
\def\Htop{black!34}

\tikzset{
  Hfig edge/.style={
    draw=black,
    line width=.55pt,
    line join=round,
    line cap=round
  },
  Hfig arrow/.style={
    draw=black,
    -{Stealth[length=3mm,width=2mm]},
    line width=.45pt
  },
  Hfig axis/.style={
    draw=black,
    -{Stealth[length=2.6mm,width=1.75mm]},
    line width=.45pt
  },
  Hfig label/.style={font=\large},
  Hfig step/.style={font=\bfseries\Large},
  Hfig recurrence/.style={font=\large},
  Hperspective/.cd,
  x/.store in=\Hcoordx,
  y/.store in=\Hcoordy,
  z/.store in=\Hcoordz
}

\tikzdeclarecoordinatesystem{Hproject}{%
  \tikzset{Hperspective/.cd,x=0,y=0,z=0,#1}%
  \pgfmathsetmacro{\Hden}{\HViewDistance+(\Hcoordz)}%
  \pgfmathsetmacro{\Hpx}{%
    (\HViewDistance*(\Hcoordx)+\HVanishingX*(\Hcoordz))/\Hden}%
  \pgfmathsetmacro{\Hpy}{%
    (\HViewDistance*(\Hcoordy)+\HVanishingY*(\Hcoordz))/\Hden}%
  \pgfpointxy{\Hpx}{\Hpy}%
}

\newcommand{\HPoint}[3]{(Hproject cs:x={#1},y={#2},z={#3})}

\newcommand{\HFrontFace}[6]{%
  \path[Hfig edge,fill=#6]
    \HPoint{#1}{#3}{#5}
    --\HPoint{#2}{#3}{#5}
    --\HPoint{#2}{#4}{#5}
    --\HPoint{#1}{#4}{#5}--cycle;
}

\newcommand{\HRightFace}[6]{%
  \path[Hfig edge,fill=#6]
    \HPoint{#1}{#2}{#4}
    --\HPoint{#1}{#2}{#5}
    --\HPoint{#1}{#3}{#5}
    --\HPoint{#1}{#3}{#4}--cycle;
}

\newcommand{\HTopFace}[6]{%
  \path[Hfig edge,fill=#6]
    \HPoint{#1}{#3}{#4}
    --\HPoint{#2}{#3}{#4}
    --\HPoint{#2}{#3}{#5}
    --\HPoint{#1}{#3}{#5}--cycle;
}

\newif\ifHpanelA

\newcommand{\HCommonSetup}{%
  \def\Hrows{15}
  \def\Hfrontwidth{5}
  \def\Hblockheight{5}
  \def\Hblocklayers{4}
  \def\Hdenseheight{5}
  \pgfmathtruncatemacro{\Hfirstk}{\Hdenseheight}%
  \pgfmathtruncatemacro{\Hterminalk}{\Hfirstk+\Hrows}%
  \pgfmathtruncatemacro{\Hlastk}{\Hterminalk+\Hblockheight}%
  \pgfmathtruncatemacro{\Hcoreend}{\Hterminalk+\Hfrontwidth}%
  \pgfmathtruncatemacro{\Hjcount}{\Hblocklayers+1}%
}

\newcommand{\HFrontBand}{%
  \pgfmathtruncatemacro{\Hbeforeterminal}{\Hterminalk-1}%
  \pgfmathtruncatemacro{\Hbottomrow}{\Hlastk-1}%
  \path[Hfig edge,fill=\Hreq]
    \HPoint{\Hfirstk}{\Hfirstk}{0}
    --\HPoint{\Hfirstk+\Hfrontwidth}{\Hfirstk}{0}
    \foreach \Hkk in {\Hfirstk,...,\Hbeforeterminal}{
      --\HPoint{\Hkk+\Hfrontwidth}{\Hkk+1}{0}
      --\HPoint{\Hkk+\Hfrontwidth+1}{\Hkk+1}{0}
    }
    --\HPoint{\Hcoreend}{\Hlastk}{0}
    \foreach \Hkk in {\Hbottomrow,...,\Hfirstk}{
      --\HPoint{\Hkk}{\Hkk+1}{0}
      --\HPoint{\Hkk}{\Hkk}{0}
    }
    --cycle;
}

\newcommand{\HLowerSlice}[1]{%
  \pgfmathtruncatemacro{\Hq}{#1}%
  \pgfmathtruncatemacro{\Hqp}{\Hq+1}%
  \ifHpanelA
    \def\Hextra{0}%
    \def\Hcolor{\Hreq}%
  \else
    \edef\Hextra{\Hq}%
    \ifnum\Hq=0
      \def\Hcolor{\Hreq}%
    \else
      \def\Hcolor{white}%
    \fi
  \fi
  \pgfmathtruncatemacro{\Houter}{\Hcoreend+\Hextra}%
  \pgfmathtruncatemacro{\Hbeforeterminal}{\Hterminalk-1}%

  \HRightFace{\Houter}{\Hterminalk}{\Hlastk}{\Hq}{\Hqp}{\Hcolor}%
  \ifHpanelA\else
    \ifnum\Hq>0
      \HFrontFace{\Houter-1}{\Houter}{\Hterminalk}{\Hlastk}{\Hq}{white}%
    \fi
  \fi

  \foreach \Hkk in {\Hterminalk,...,\Hfirstk}{%
    \pgfmathtruncatemacro{\Hright}{\Hkk+\Hfrontwidth+\Hextra}%
    \ifnum\Hkk<\Hterminalk
      \HRightFace{\Hright}{\Hkk}{\Hkk+1}{\Hq}{\Hqp}{\Hcolor}%
    \fi
    \ifnum\Hkk>\Hfirstk
      \HTopFace{\Hright-1}{\Hright}{\Hkk}{\Hq}{\Hqp}{\Hcolor}%
    \else
      \ifnum\Hq=\Hjcount
        \HTopFace{\Hfirstk}{\Hright}{\Hkk}{\Hq}{\Hqp}{white}%
      \fi
    \fi
    \ifHpanelA\else
      \ifnum\Hq>0
        \ifnum\Hkk<\Hterminalk
          \HFrontFace{\Hright-1}{\Hright}{\Hkk}{\Hkk+1}{\Hq}{white}%
        \fi
      \fi
    \fi
  }%
  \ifnum\Hq=0
    \HFrontBand
  \fi
}

\newcommand{\HDenseSlice}[1]{%
  \pgfmathtruncatemacro{\Hq}{#1}%
  \pgfmathtruncatemacro{\Hqp}{\Hq+1}%
  \pgfmathtruncatemacro{\Hdenselast}{\Hdenseheight-1}%
  \foreach \Hkk in {\Hdenselast,...,0}{%
    \ifnum\Hkk=0
      \ifHpanelA
        \HRightFace{\Hcoreend}{0}{1}{\Hq}{\Hqp}{\Hside}%
      \else
        \ifnum\Hq=0
          \HRightFace{\Hcoreend}{0}{1}{\Hq}{\Hqp}{\Hside}%
        \else
          \HRightFace{\Hcoreend+\Hq}{0}{1}{\Hq}{\Hqp}{white}%
        \fi
      \fi
      \HTopFace{0}{\Hcoreend}{0}{\Hq}{\Hqp}{\Htop}%
      \ifHpanelA\else
        \ifnum\Hq>0
          \HTopFace{\Hcoreend}{\Hcoreend+\Hq}{0}{\Hq}{\Hqp}{white}%
          \HFrontFace{\Hcoreend+\Hq-1}{\Hcoreend+\Hq}{0}{1}{\Hq}{white}%
        \fi
      \fi
    \else
      \HRightFace{\Hcoreend}{\Hkk}{\Hkk+1}{\Hq}{\Hqp}{\Hside}%
    \fi
    \ifnum\Hq=0
      \HFrontFace{\Hkk}{\Hcoreend}{\Hkk}{\Hkk+1}{0}{\Hfront}%
    \fi
  }%
}

\newcommand{\HDrawStructure}{%
  \HCommonSetup
  \ifHpanelA
    \pgfmathtruncatemacro{\Hlayers}{\Hblocklayers}%
  \else
    \pgfmathtruncatemacro{\Hlayers}{\Hblocklayers+1}%
  \fi
  \foreach \Hslice in {\Hlayers,...,0}{%
    \HLowerSlice{\Hslice}%
    \ifnum\Hslice<\Hjcount
      \HDenseSlice{\Hslice}%
    \fi
  }%
}

\newcommand{\HAxes}{%
  \draw[Hfig axis]
    \HPoint{0}{17}{0}--\HPoint{0}{17}{-1.8}
    node[Hfig label,anchor=north east] {$j$};
  \draw[Hfig axis]
    \HPoint{0}{17}{0}--\HPoint{3.6}{17}{0}
    node[Hfig label,anchor=west] {$\ell$};
  \draw[Hfig axis]
    \HPoint{0}{17}{0}--\HPoint{0}{21}{0}
    node[Hfig label,anchor=north] {$k$};
}

\newcommand{\HRecurrenceAnnotations}{%
  \node[Hfig step] at (18.2,-8.0) {step 1};
  \node[Hfig recurrence] at (18.2,-5.7)
    {$H_{j,0,\cdot}\longrightarrow H_{j+1,0,\cdot}$};
  \draw[Hfig arrow]
    \HPoint{12}{0}{4.4}--\HPoint{12}{0}{.7};

  \node[Hfig step] at (40.0,-3.6) {step 2};
  \node[Hfig recurrence,anchor=west] at (32.8,-.9)
    {$H_{\cdot,k,\cdot}\longrightarrow H_{\cdot,k+1,\cdot}$};
  \draw[Hfig arrow]
    \HPoint{32}{.5}{0}--\HPoint{32}{4.0}{0};

  \node[Hfig step,anchor=west] at (38.4,8.6) {step 3};
  \node[Hfig recurrence,anchor=west] at (35.0,11.7)
    {$H_{0,\cdot,\ell}\longrightarrow H_{0,\cdot,\ell+1}$};
  \node[Hfig recurrence,anchor=west] at (35.0,14.0)
    {$H_{1,\cdot,\ell}\longrightarrow H_{1,\cdot,\ell+1}$};
  \draw[Hfig arrow]
    \HPoint{30.5}{10.0}{0}--\HPoint{34.0}{10.0}{0};

  \node[Hfig step,anchor=west] at (35.4,23.0) {step 4};
  \node[Hfig recurrence,anchor=west] at (32.8,26.2)
    {$H_{j,\cdot,\cdot}\longrightarrow H_{j+1,\cdot,\cdot}$};
  \draw[Hfig arrow]
    \HPoint{27}{28}{5.0}--\HPoint{27}{28}{1.3};
}

\newcommand{\HRequiredMomentsTikZ}{%
  \begin{tikzpicture}[x=.22cm,y=-.22cm]
    \HpanelAtrue
    \HDrawStructure
    \HAxes
  \end{tikzpicture}%
}

\newcommand{\HRecurrenceClosedTikZ}{%
  \begin{tikzpicture}[x=.22cm,y=-.22cm]
    \HpanelAfalse
    \HDrawStructure
    \HRecurrenceAnnotations
  \end{tikzpicture}%
}

\newcommand{\HRequiredMomentsPanel}{%
  \raisebox{0.45cm}{%
    \resizebox{\HPanelAWidth}{!}{\HRequiredMomentsTikZ}%
  }%
}

\newcommand{\HRecurrenceClosedPanel}{%
  \resizebox{\HPanelBWidth}{!}{\HRecurrenceClosedTikZ}%
}

\providecommand{\weakfastfigwidth}{0.92\linewidth}

\tikzset{
  StairEdge/.style={draw=black,line width=.55pt,line join=round,line cap=round},
  StairArrow/.style={draw=black,-{Stealth[length=3mm,width=2mm]},line width=.45pt},
  StairLabel/.style={font=\large},
  StairStep/.style={font=\bfseries\Large},
  StairRecurrence/.style={font=\large}
}

\providecommand{\WeakFastViewDistance}{30}
\providecommand{\WeakFastVanishingX}{48}
\providecommand{\WeakFastVanishingY}{-18}

\tikzset{
  WeakFastProjection/.cd,
  x/.store in=\WeakFastCoordX,
  k/.store in=\WeakFastCoordK,
  q/.store in=\WeakFastCoordQ
}

\tikzdeclarecoordinatesystem{WeakFastProject}{%
  \tikzset{WeakFastProjection/.cd,x=0,k=0,q=0,#1}%
  \pgfmathsetmacro{\WeakFastDen}{\WeakFastViewDistance+(\WeakFastCoordQ)}%
  \pgfmathsetmacro{\WeakFastPX}{%
    (\WeakFastViewDistance*(\WeakFastCoordX)
      +\WeakFastVanishingX*(\WeakFastCoordQ))/\WeakFastDen}%
  \pgfmathsetmacro{\WeakFastPY}{%
    -(\WeakFastViewDistance*(\WeakFastCoordK)
      +\WeakFastVanishingY*(\WeakFastCoordQ))/\WeakFastDen}%
  \pgfpointxy{\WeakFastPX}{\WeakFastPY}%
}

\newcommand{\WeakFastPoint}[3]{%
  (WeakFastProject cs:x={#1},k={#2},q={#3})%
}

\newcommand{\WeakFastDrawLambdaBlock}{%
  \def\WeakFastRows{17}%
  \def\WeakFastDepthDivisions{5}%
  \pgfmathtruncatemacro{\WeakFastLastRow}{\WeakFastRows-1}%
  \pgfmathtruncatemacro{\WeakFastLastDivision}{\WeakFastDepthDivisions-1}%

  \foreach \r in {0,...,\WeakFastLastRow}{%
    \path[fill=\BaseFill,draw=none]
      \WeakFastPoint{\r}{\r}{0}
      --\WeakFastPoint{\WeakFastRows}{\r}{0}
      --\WeakFastPoint{\WeakFastRows}{\r+1}{0}
      --\WeakFastPoint{\r}{\r+1}{0}--cycle;
  }%
  \foreach \r in {0,...,\WeakFastLastRow}{%
    \path[fill=\BaseFill,draw=none]
      \WeakFastPoint{\WeakFastRows}{\r}{0}
      --\WeakFastPoint{\WeakFastRows}{\r}{\WeakFastDepthDivisions}
      --\WeakFastPoint{\WeakFastRows}{\r+1}{\WeakFastDepthDivisions}
      --\WeakFastPoint{\WeakFastRows}{\r+1}{0}--cycle;
  }%
  \path[fill=\BaseFill,draw=none]
    \WeakFastPoint{0}{0}{0}
    --\WeakFastPoint{\WeakFastRows}{0}{0}
    --\WeakFastPoint{\WeakFastRows}{0}{\WeakFastDepthDivisions}
    --\WeakFastPoint{0}{0}{\WeakFastDepthDivisions}--cycle;

  \draw[StairEdge]
    \WeakFastPoint{0}{0}{0}
    --\WeakFastPoint{\WeakFastRows}{0}{0}
    --\WeakFastPoint{\WeakFastRows}{0}{\WeakFastDepthDivisions};
  \draw[StairEdge]
    \WeakFastPoint{0}{0}{0}
    --\WeakFastPoint{0}{0}{\WeakFastDepthDivisions}
    --\WeakFastPoint{\WeakFastRows}{0}{\WeakFastDepthDivisions};

  \foreach \r in {1,...,\WeakFastLastRow}{%
    \draw[StairEdge]
      \WeakFastPoint{\r-1}{\r}{0}
      --\WeakFastPoint{\WeakFastRows}{\r}{0}
      --\WeakFastPoint{\WeakFastRows}{\r}{\WeakFastDepthDivisions};
  }%
  \draw[StairEdge]
    \WeakFastPoint{\WeakFastLastRow}{\WeakFastRows}{0}
    --\WeakFastPoint{\WeakFastRows}{\WeakFastRows}{0};

  \foreach \r in {0,...,\WeakFastLastRow}{%
    \draw[StairEdge]
      \WeakFastPoint{\r}{\r}{0}--\WeakFastPoint{\r}{\r+1}{0};
  }%

  \draw[StairEdge]
    \WeakFastPoint{\WeakFastRows}{0}{0}
    --\WeakFastPoint{\WeakFastRows}{\WeakFastRows}{0};
  \draw[StairEdge]
    \WeakFastPoint{\WeakFastRows}{0}{\WeakFastDepthDivisions}
    --\WeakFastPoint{\WeakFastRows}{\WeakFastRows}{\WeakFastDepthDivisions};
  \draw[StairEdge]
    \WeakFastPoint{\WeakFastRows}{\WeakFastRows}{\WeakFastDepthDivisions}
    --\WeakFastPoint{\WeakFastRows}{\WeakFastRows}{0};

  \foreach \q in {1,...,\WeakFastLastDivision}{%
    \draw[StairEdge]
      \WeakFastPoint{0}{0}{\q}
      --\WeakFastPoint{\WeakFastRows}{0}{\q}
      --\WeakFastPoint{\WeakFastRows}{\WeakFastRows}{\q};
  }%
}

\newcommand{\WeakFastDrawLambdaAnnotations}{%
  \draw[StairArrow]
    \WeakFastPoint{22.9}{3.1}{0}--\WeakFastPoint{22.9}{10.8}{0};
  \node[StairStep,anchor=west] at (24.1,-7.0) {step 2};
  \node[StairRecurrence,anchor=west] at (23.7,-9.8)
    {$\tilde{\Lambda}_{\cdot,k,\cdot}\to\tilde{\Lambda}_{\cdot,k+1,\cdot}$};

  \begin{scope}[shift={(-3.5,3)}]
  \coordinate (Oright) at (WeakFastProject cs:x=4.4,k=12.8,q=0);
  \draw[StairArrow] (Oright)--\WeakFastPoint{9.2}{12.8}{0};
  \draw[StairArrow] (Oright)--\WeakFastPoint{4.4}{17.7}{0};
  \draw[StairArrow] (Oright)--\WeakFastPoint{4.4}{12.8}{-2};
  \node[StairLabel] at (10.25,-12.4) {$\ell$};
  \node[StairLabel] at (5.85,-17.35) {$k$};
  \node[StairLabel] at (0.3,-15.05) {$i$};
  \end{scope}
}

\newcommand{\WeakFastRecurrenceTikZ}{%
    \begin{tikzpicture}[x=0.28cm,y=0.28cm]
      \def\BaseFill{black!34}

      \begin{scope}[shift={(0,0)}]
        \foreach \j in {0,...,5}{
          \path[fill=\BaseFill,draw=none] (0,-\j) rectangle (14,-\j-1);
        }
        \foreach \j in {0,...,4}{
          \pgfmathtruncatemacro{\rowend}{19-\j}
          \path[fill=white,draw=none] (14,-\j) rectangle (\rowend,-\j-1);
        }

        \foreach \j in {0,...,5}{
          \pgfmathtruncatemacro{\rowend}{19-\j}
          \draw[StairEdge] (0,-\j) -- (\rowend,-\j);
        }
        \draw[StairEdge] (0,-6) -- (14,-6);
        \draw[StairEdge] (0,0) -- (0,-6);
        \draw[StairEdge] (14,0) -- (14,-6);
        \draw[StairEdge]
          (19,0) -- (19,-1) -- (18,-1) -- (18,-2)
          -- (17,-2) -- (17,-3) -- (16,-3) -- (16,-4)
          -- (15,-4) -- (15,-5) -- (14,-5);
        \foreach \j in {0,...,5}{
          \pgfmathtruncatemacro{\rowend}{19-\j}
          \draw[StairEdge] (0,-\j) -- (\rowend,-\j);
        }

        \draw[StairArrow] (-0.85,-0.7) -- (-0.85,-4.7);
        \node[StairStep,anchor=east] at (-1.75,-1.45) {step 1};
        \node[StairRecurrence,anchor=east] at (-1.0,-4.00)
          {$\tilde{H}_{j,0,\cdot}\to\tilde{H}_{j+1,0,\cdot}$};

        \coordinate (Oleft) at (1.0,-8.4);
        \draw[StairArrow] (Oleft) -- (4.7,-8.4);
        \draw[StairArrow] (Oleft) -- (1.0,-11.2);
        \node[StairLabel] at (5.4,-8.05) {$\ell$};
        \node[StairLabel] at (0.35,-11.8) {$j$};
      \end{scope}

      \draw[StairArrow] (22.4,-2.6) -- (34.8,-2.6);
      \node[StairRecurrence] at (28.6,-1.1) {contraction in $j$};
      \node[StairRecurrence] at (28.6,-5.75)
        {$\tilde{\Lambda}_{i,0,\ell}=\displaystyle\sum_{j=0}^{m}C_{ij}\tilde{H}_{j,0,\ell}$};

      \begin{scope}[shift={(39,0)}]
        \WeakFastDrawLambdaBlock
        \WeakFastDrawLambdaAnnotations
      \end{scope}
    \end{tikzpicture}%
}

\newcommand{\WeakFastRecurrenceGraphic}{%
  \resizebox{\weakfastfigwidth}{!}{\WeakFastRecurrenceTikZ}%
}

\Crefname{equation}{}{}
\renewcommand{\thesubfigure}{\Alph{subfigure}} 

\newtheorem{theorem}{Theorem}[section]
\newtheorem{lemma}{Lemma}[section]
\theoremstyle{definition}

\newtheorem{corollary}[theorem]{Corollary}

\crefname{corollary}{corollary}{corollaries}
\Crefname{corollary}{Corollary}{Corollaries}

\theoremstyle{remark}

\def\mA{\mathcal{A}}
\def\mB{\mathcal{B}}
\def\mC{\mathcal{C}}
\def\mD{\mathcal{D}}

\def\mG{\mathcal{G}}

\def\mI{\mathcal{I}}

\def\mM{\mathcal{M}}

\def\mO{\mathcal{O}}

\def\mQ{\mathcal{Q}}
\def\mR{\mathcal{R}}
\def\mS{\mathcal{S}}

\def\mV{\mathcal{V}}

\def\md{\mathrm{d}}

\def\be{\boldsymbol{e}_0}
\def\bf{\boldsymbol{f}}

\def\bu{\boldsymbol{u}}

\def\bv{\boldsymbol{v}}

\graphicspath{{figures/}}
\begin{document}

\title[Fast spectral methods for Volterra integral equations]{Fast spectral methods for Volterra integral equations}

\author[1]{\fnm{Xi} \sur{Chen}}\email{chen.xi@mail.ustc.edu.cn}
\author[1]{\fnm{Kuan} \sur{Xu}}\email{kuanxu@ustc.edu.cn}

\affil[1]{\orgdiv{School of Mathematical Sciences}, \orgname{University of Science and Technology of China}, \orgaddress{\street{96 Jinzhai Road}, \city{Hefei}, \postcode{230026}, \state{Anhui}, \country{China}}}



\abstract{\unboldmath\mathversion{normal}
We develop novel methods for the fast construction of spectral approximations to smooth linear and weakly singular Volterra integral operators and for the fast evaluation of smooth nonlinear Volterra integrals. The proposed methods employ the Duffy mapping, which transforms the original indefinite integral into a definite integral and maps the original causal triangular domain to a square, thereby avoiding any artificial extension of the kernel function. For smooth linear and weakly singular Volterra operators, the resulting approximations are almost-banded and lower-banded, respectively. The $N\times N$ principal finite section of these approximations can be constructed in $\mO(N)$ and $\mO(N^2)$ flops, respectively, via recurrence relations. For Volterra integrals with a smooth kernel that depends nonlinearly on a given function, recurrence relations lead to fast evaluation with a complexity of $\mO(N^2\log N)$ when the function is approximated by a degree-$N$ Chebyshev approximant. These spectral approximations and evaluation procedures give rise to fast and stable spectral methods for solving Volterra integral equations, as well as a broad class of problems involving Volterra integral operators. Extensive numerical experiments show that these spectral methods significantly outperform existing methods in terms of stability, speed, accuracy, and ease of implementation.}

\maketitle

\section{Introduction}\label{sec:intro}
In this paper, we develop fast algorithms for constructing spectral approximations to, or evaluating, Volterra integral operators, and based on these algorithms, devise fast spectral methods for solving Volterra integral equations that significantly outperform existing methods. The Volterra integral operators considered in the remainder of this paper include:
\begin{itemize}
\item \emph{smooth linear Volterra operator},
\begin{align}\label{eq:smooth_volterra}
(\mV u)(x)
=
\int_{-1}^{x} K(x,y)u(y)\,\md y,
\qquad -1\leqslant x\leqslant 1,
\end{align}
where the kernel $K(x,y)$ is a smooth function. Smooth linear Volterra operators commonly arise in linear Volterra integral and integro-differential equations. They also occur as linearized operators or as building blocks in more complicated mathematical models. In such applications, the same operator may need to be assembled, applied, or inverted, making an efficient spectral approximation particularly useful.

\item \emph{smooth nonlinear Volterra integral operator},
\begin{align}\label{eq:nonlinear_volterra}
(\mV u)(x)
=
\int_{-1}^{x} K\bigl(x,y,u(y)\bigr)\,\md y,
\qquad -1\leqslant x\leqslant 1,
\end{align}
where the integrand depends nonlinearly on $u(y)$. Such operators, also known as nonlinear Volterra integral operators of Urysohn type, arise in nonlinear Volterra integral and integro-differential equations.

\item \emph{weakly singular Volterra operator},
\begin{align}\label{eq:singular_volterra}
(\mV u)(x)
=
\int_{-1}^{x}
(x-y)^{-\mu}K(x,y)u(y)\,\md y,
\qquad  -1\leqslant x\leqslant 1,
\end{align}
where the kernel is the product of a smooth factor $K(x,y)$ and an integrable algebraic singularity of strength $\mu \in (0,1)$ at $x=y$. Operators of this kind include Abel-type and fractional integral operators.
\end{itemize}

The primary difficulty in approximating these operators is that their kernel functions are defined on the triangular causal domain $T^2=\{(x,y):-1\leqslant y\leqslant x\leqslant 1\}$ rather than on the canonical square $[-1,1]^2$. Thus, the success of spectral methods for either constructing approximations to Volterra integral operators or solving Volterra integral equations hinges on how accurately, efficiently, and stably the kernel function can be approximated. Moreover, the nonlinear and singular natures of \cref{eq:nonlinear_volterra,eq:singular_volterra}, respectively, make the problem even more challenging. For a nonlinear Volterra integral operator, the function $u(x)$ becomes part of the kernel, whereas for a weakly singular Volterra operator, the singular factor generally induces fractional endpoint behavior that cannot be effectively represented by ordinary polynomial bases.

Early works on solving smooth Volterra integral equations primarily employed collocation-based methods \cite{tan,she2}, resulting in dense matrices. Both the construction and solution of the resulting linear systems incur $\mO(N^3)$ complexity, which makes these methods impractical for large-scale problems. When the kernel is convolutional, i.e., when the variables $x$ and $y$ are coupled only through their difference $x-y$, fast spectral methods \cite{xu,hal1} are available via sparse spectral approximations to Volterra-type convolution operators. However, these methods do not apply to equations involving the general linear Volterra operator \cref{eq:smooth_volterra}. In \cite{hal1}, Hale also proposes a coefficient-based spectral method for general Volterra integral equations by assuming that the kernel is defined on the non-causal triangle $\tilde{T}^2 = \{(x,y):-1\leqslant x\leqslant y \leqslant 1\}$ by the same closed-form expression. Under this assumption, the kernel function can be approximated on $[-1, 1]^2$ by a bivariate polynomial approximant. With the same assumption, the pseudospectral method proposed by Driscoll \cite{dri,chebfun} reduces the construction complexity to $\mO(N^2)$ compared with the methods in \cite{tan,she2}. However, such an assumption may be neither reasonable nor viable---there are plenty of examples in which the kernel function is singular or discontinuous in $\tilde{T}^2$. See \cref{sec:example_2} for an example. More recently, Gutleb and Olver \cite{gut1} proposed a sparse spectral method that approximates the kernel function by bivariate polynomials orthogonal on $T^2$, followed by the construction of an approximation to the Volterra operator via an operator-valued Clenshaw algorithm. For a kernel approximant of degree $n$ and a polynomial solution of degree $N$, the resulting matrix approximation to a smooth linear Volterra operator is an infinite-dimensional banded matrix with bandwidths $(n,n)$, for which the construction and solution complexities are $\mO(n^3N)$ and $\mO(n^2N)$, respectively. Hence, the method is significantly faster than collocation-based spectral methods. However, as shown in \cref{sec:example_3}, the operator-valued Clenshaw recurrence is numerically unstable, making the method unsuitable for problems involving complicated kernel functions.

Nonlinear Volterra integral equations are usually solved via Newton's method, where the main computational tasks are solving the linearized Volterra operator and evaluating the residuals. Since the former involves the Fr\'{e}chet derivative of the nonlinear Volterra integral operator with respect to the function $u(y)$, the resulting linearized operator can be rather complicated. One exception is the nonlinear Volterra integral operator of Hammerstein type:
\begin{align}\label{eq:nonlinear_volterra_subcase}
(\mV u)(x) = \int_{-1}^{x} K\bigl(x,y\bigr)F\bigl(y,u(y)\bigr)\,\md y, \qquad -1\leqslant x\leqslant 1,
\end{align}
where $F\bigl(y,u(y)\bigr)$ is a nonlinear function of $u(y)$. This explains why state-of-the-art methods, such as those in \cite{chebfun,gut2}, focus only on this much simpler case. When used to solve nonlinear Volterra equations, the \textsc{Chebfun} routine \texttt{volt} can be regarded as a nonlinear extension of Driscoll's method \cite{dri} and has a complexity of $\mO(N^3)$ for inversion. Its limitation again lies in the assumption that the closed-form expression on $T^2$ also applies to $\tilde{T}^2$, which is not always the case. The method proposed in \cite{gut2} is a natural extension of \cite{gut1} and is likewise subject to the instability issue. When used to evaluate a nonlinear Volterra integral, the \textsc{Chebfun} routine \texttt{volt} instead employs a conventional but stable approach that evaluates \cref{eq:nonlinear_volterra} by repeated quadrature with upper limits given by a set of interpolation points, such as Chebyshev points \cite{xu2}, at a cost of $\mO(N^3)$. This conventional method serves as the baseline in \cref{sec:nonlinear_volt}.

Solving a Volterra integral equation involving weakly singular kernel functions is difficult because the solution may be singular. Moreover, the solution may also be smooth, with its regularity unknown in advance. Early attempts to solve weakly singular Volterra integral equations include the collocation method \cite{che1}, which reformulates the original equation so that the solution to the reformulated equation is somewhat smoother. However, the convergence is usually only algebraic. The Galerkin spectral method proposed in \cite{li} embeds a mapping into the solution so that the mapped solution is sufficiently smooth to achieve spectral convergence. This method, however, applies only to singular Volterra integral equations with $\mu = p/q$, where $\mu$ is rational and $q$ is not large. For large $q$, the method suffers from numerical instability. More recently, the collocation spectral method proposed in \cite{hou} employs Jacobi fractional polynomials, also known as M\"{u}ntz polynomials, to approximate the functions involved. Since the method produces dense collocation matrices whose entries are obtained by quadrature, a cost of $\mO(N^3)$ for both construction and solution is inevitable.

The fast and numerically stable algorithms proposed in this paper either construct sparse spectral approximations to smooth linear and weakly singular Volterra integral operators or provide fast evaluation of smooth nonlinear Volterra integrals. In contrast to many existing methods, the proposed methods do not sample the kernel function in the non-causal triangle. By means of the so-called Duffy transform, the indefinite integrals \cref{eq:smooth_volterra,eq:nonlinear_volterra,eq:singular_volterra} are transformed into definite integrals in which $u(y)$ is restricted to a subinterval of $[-1,1]$. Meanwhile, the transformed kernel is well defined on $[-1,1]^2$ and can therefore be approximated by standard bivariate Chebyshev approximants. Various recurrence relations are derived for the resulting construction and evaluation procedures, leading to computational complexities significantly lower than those of existing methods. Furthermore, the methods are straightforward to implement. Although Jacobi polynomials play a role in the derivation, the actual computation involves only Chebyshev polynomials and points. Hence, no value-to-coefficient transform other than the discrete cosine transform (DCT) is required.

Let $n$ denote the highest degree of the polynomials used in the bivariate Chebyshev approximants for the transformed kernel functions. For the smooth linear Volterra operator, the resulting spectral approximation is an almost-banded matrix of infinite dimensions, for which the construction and inversion costs are $\mO(n^2\log n N)$ and $\mO(n^2 N)$, respectively. The proposed approximation to weakly singular linear Volterra operators applies to \cref{eq:singular_volterra} for any $\mu$ without requiring $\mu$ to be rational. The resulting infinite-dimensional approximation is lower-banded with lower bandwidth $n+p$, where $p = (1-\mu)/\gamma$ is a small integer. Hence, both the construction and solution costs are $\mO(n N^2)$. These spectral approximations lead to fast and well-conditioned spectral methods for Volterra integral equations with smooth or weakly singular kernels. The proposed method for evaluating smooth nonlinear Volterra integrals applies to the more general Urysohn class, with complexity $\mO(N^2 \log N)$. The Picard iteration, equipped with this fast evaluation procedure, provides an effective and efficient solver for nonlinear Volterra integral equations.

In the rest of this paper, we say that a matrix $L$ has bandwidths $(\xi_l, \xi_u)$ if its entries satisfy $L_{ij}=0$ for $i-j>\xi_l$ or $j-i>\xi_u$. We use $\mD$, $\mS$, and $\mM_k[f]$ to denote the infinite-dimensional matrices that approximate the differentiation, conversion, and multiplication operators, respectively \cite{olv}. More precisely, $\mD$ maps the Chebyshev coefficients of a given function to the $C^{(1)}$ ultraspherical coefficients of its derivative, while $\mS$ converts the Chebyshev coefficients of the same function to its $C^{(1)}$ ultraspherical coefficients. The matrix $\mM_k[f]$ represents multiplication by $f(x)$ or, more precisely, multiplication by the approximant to $f(x)$ in the $C^{(k)}$ ultraspherical basis, with $C^{(0)}$ understood as the Chebyshev polynomials of the first kind. When $f(x)$ is a polynomial of degree one, $\mM_0[f]$ is tridiagonal. The Kronecker delta is defined by $\delta_{ij}=1$ when $i=j$ and $\delta_{ij}=0$ otherwise. The Pochhammer symbol, also known as the rising factorial, is defined by $(a)_n=a(a+1)\cdots(a+n-1)$, with $(a)_0=1$. In addition, we shall frequently use the generalized Jacobi polynomials $P_n^{(-1, \beta)}(x)$ and $P_n^{(\beta, -1)}(x)$, which are defined by \cref{eq:jacobi_parity,eq:jacobi_minus_one_zero,eq:jacobi_minus_one_left}. Any polynomial of negative degree is taken to be zero.

In \Cref{sec:restrict}, we study the restriction operator, which is the key tool for constructing approximations to Volterra integral operators of all three types. Fast and stable methods for constructing spectral approximations to, or rapidly evaluating, \cref{eq:smooth_volterra,eq:nonlinear_volterra,eq:singular_volterra} are proposed in \Cref{sec:smooth_volt,sec:nonlinear_volt,sec:weaksingular_volt}, respectively. \Cref{sec:numerical_experiments} presents numerical experiments demonstrating the accuracy and efficiency of the proposed methods. The identities of Chebyshev and Jacobi polynomials used throughout the paper are collected in \Cref{sec:appendix}.

\section{Restriction operator}\label{sec:restrict}
For a smooth function $u(x)$ defined on $[-1,1]$, the restriction operator $\mC_{[a,b]}$ is defined by
\begin{align}
(\mC_{[a,b]}u)(x) = u\left( \frac{b-a}{2}x+\frac{b+a}{2} \right), \qquad x\in[-1,1],
\end{align}
where $[a,b]\subseteq[-1,1]$. In other words, the restriction operator takes the part of $u(x)$ defined on $[a,b]$ and stretches it to $[-1,1]$. It is linear and multiplicative:
\begin{align}\label{eq:restriction_linear_multiplicative}
\mC_{[a,b]}(\lambda_1 u+ \lambda_2v) &= \lambda_1\mC_{[a,b]}u+\lambda_2\mC_{[a,b]}v, \qquad \mC_{[a,b]}(uv) = (\mC_{[a,b]}u)(\mC_{[a,b]}v).
\end{align}

In this section, we construct a spectral approximation to the restriction operator using Chebyshev polynomials of the first kind as the basis. This infinite-dimensional matrix approximation will be used to develop fast algorithms for the Volterra integral operators \cref{eq:smooth_volterra,eq:nonlinear_volterra,eq:singular_volterra}.

\subsection{Recursive construction of the restriction operator}
Applying the restriction operator to the quasimatrix $\mathbf{T}(x) = \bigl(T_0(x),T_1(x),T_2(x),\ldots\bigr)$
gives
\begin{align}
\left(\mC_{[a,b]}\mathbf{T}\right)(x) = \mathbf{T}(x)\mathcal{R}(a,b),
\end{align}
where $\mathcal{R}(a,b)$ is an infinite-dimensional matrix that depends on $a$ and $b$. Let $R_{ij}(a,b)$ denote the $(i, j)$th entry of $\mathcal{R}(a,b)$. The orthogonality of the Chebyshev polynomials \cref{eq:chebyshev_orthogonality} leads to
\begin{equation}
R_{ij}(a,b)
    =
    \frac{2-\delta_{i0}}{\pi}
    \int_{-1}^1
    T_i(x)
    T_j\left(
        \frac{b-a}{2}x+\frac{b+a}{2}
    \right)
    \frac{\md x}{\sqrt{1-x^2}}.
\end{equation}
Since $T_0(x)=1$ and $T_1(x)=x$, the nonzero elements in the first two columns of $\mathcal R(a,b)$ are
\begin{align*}
R_{00}(a,b) = 1,\qquad R_{01}(a,b)=\frac{a+b}{2},\qquad R_{11}(a,b)=\frac{b-a}{2}.
\end{align*}
The $j$th column for $j \geqslant 2$ can be derived from a recurrence relation. 

\begin{theorem}\label{thm:restriction_recursion}
For $-1 \leqslant a < b \leqslant 1$, $\mathcal{R}(a,b)$ is upper triangular, and for $j\geqslant 1$ its columns satisfy
\begin{align}
R_{:,j+1}(a,b)
        &=
        2\mM_0
        \left[
            \frac{b-a}{2}x+\frac{b+a}{2}
        \right]
        R_{:,j}(a,b)
        -
        R_{:,j-1}(a,b).
        \label{eq:restriction_column_recursion_chebyshev}
\end{align}
Or, equivalently, for $j\geqslant 1$ the nonzero entries satisfy a four- or five-term recurrence relation
\begin{subequations}\label{eq:five_point_recurrence}
\begin{align}
        R_{0,j+1}
        &=
        \frac{b-a}{2}R_{1,j}
        +(a+b)R_{0,j}
        -R_{0,j-1},
        \\
        R_{1,j+1}
        &=
        (b-a)R_{0,j}
        +\frac{b-a}{2}R_{2,j}
        +(a+b)R_{1,j}
        -R_{1,j-1},
        \\
        R_{i,j+1}
        &=
        \frac{b-a}{2}
        \left(
            R_{i-1,j}+R_{i+1,j}
        \right)
        +(a+b)R_{i,j}
        -R_{i,j-1},
        \qquad 2\leqslant i\leqslant j+1,
\end{align}
\end{subequations}
where the arguments $(a,b)$ are omitted for notational simplicity.
\end{theorem}

\begin{proof}
Since
\begin{align}\label{eq:restriction_column_definition}
        \left(\mC_{[a,b]}T_j\right)(x)
        =
        T_j\left(
            \frac{b-a}{2}x+\frac{b+a}{2}
        \right)
\end{align}
is a polynomial in $x$ of degree at most $j$, $R_{ij}(a,b)=0$ for $i>j$. Thus, $\mathcal{R}(a,b)$ is upper triangular.

Applying the recurrence relation \cref{eq:chebyshev_three_term} to \cref{eq:restriction_column_definition} and using the multiplicative property in \cref{eq:restriction_linear_multiplicative}, we obtain a recurrence relation for $\left(\mC_{[a,b]}T_j\right)(x)$. Taking the Chebyshev coefficients of both sides then gives \cref{eq:restriction_column_recursion_chebyshev}.

The multiplication matrix $\mM_0[(b-a)x/2+(b+a)/2]$ is tridiagonal, with main diagonal and superdiagonal entries being $(a+b)/2$ and $(b-a)/4$, respectively. The subdiagonal entries are $(b-a)/4$, except that the $(1, 0)$th entry is $(b-a)/2$. The entrywise recurrence relations \cref{eq:five_point_recurrence} follow from matching the nonzero entries of \cref{eq:restriction_column_recursion_chebyshev} rowwise.
\end{proof}

\Cref{thm:restriction_recursion} suggests that $\mathcal{R}(a,b)$ can be constructed columnwise. For the leading $N \times N$ section of $\mathcal{R}(a,b)$, this incurs an $\mO(N^2)$ complexity, which is asymptotically optimal. 

One application of the spectral approximation to the restriction operator is the solution of Volterra integral equations with proportional delay; see \cref{sec:example_4}.

\subsection{Analytic representations for $a=-1$ or $b = 1$}
Two important cases are $a=-1$ and $b = 1$. In fact, the entries of the spectral approximations to $\mC_{[-1,t]}$ and $\mC_{[t,1]}$ admit closed-form expressions.

\begin{theorem}\label{thm:restriction_analytical_expression}
For $t \in (-1,1]$, the nonzero entries of $\mR(-1, t)$ are given by
\begin{subequations}
\begin{align}
R_{00}(-1,t) &= 1, \label{eq:restriction_analytical_expression_00-1}\\
R_{ij}(-1,t) &= (2-\delta_{i0}) \frac{j}{i+j} \left(\frac{t+1}{2}\right)^i P_{j-i}^{(-1,\,2i)}(t), \qquad 0\leqslant i\leqslant j,\quad j\geqslant 1. \label{eq:restriction_analytical_expression_neg1_t}
\end{align}
\end{subequations}
For $t \in [-1,1)$, the nonzero entries of $\mR(t, 1)$ are given by
\begin{subequations}
\begin{align}
R_{00}(t,1) &= 1, \label{eq:restriction_analytical_expression_00+1}\\
R_{ij}(t,1) &= (2-\delta_{i0}) \frac{j}{i+j} \left(\frac{1-t}{2}\right)^i P_{j-i}^{(2i,\,-1)}(t), \qquad 0\leqslant i\leqslant j,\quad j\geqslant1.\label{eq:restriction_analytical_expression_t_1}
\end{align}
\end{subequations}
\end{theorem}

\begin{proof}
Because $(\mC_{[a,b]}T_0)(x)=T_0(x)$, \cref{eq:restriction_analytical_expression_00-1,eq:restriction_analytical_expression_00+1} hold. 

For $j\geqslant 1$, it follows from \cref{eq:chebyshev_orthogonality} that
\begin{align*}
R_{ij}(-1,t) = \frac{2-\delta_{i0}}{\pi} \int_{-1}^{1} T_i(x) T_j\left( \frac{t+1}{2}x+\frac{t-1}{2} \right) \frac{\md x}{\sqrt{1-x^2}}.
\end{align*}
For $i=0$, the hypergeometric representation of $T_j(x)$ and termwise integration using the beta integral give
\begin{align*}
\int_{-1}^{1} T_j\left(\frac{t+1}{2}x+\frac{t-1}{2} \right)
    \frac{\md x}{\sqrt{1-x^2}}
    =
    \pi(-1)^j
    {}_2F_1\left(
        -j,j;1;\frac{t+1}{2}
    \right)=
    \pi P_j^{(-1,0)}(t),
\end{align*}
where the last equality follows from the hypergeometric representation of the Jacobi polynomial~\cite[\S IV]{sze}. For $1\leqslant i\leqslant j$, applying \cref{eq:chebyshev_rodrigues}, integrating by parts $i$ times, and using \cref{eq:chebyshev_derivative} leads to
\begin{align*}
\int_{-1}^{1}
        T_i(x)
        T_j\left(
            \frac{t+1}{2}x+\frac{t-1}{2}
        \right)
        \frac{\md x}{\sqrt{1-x^2}}
        =
        \frac{2^{2i-1}i!(i-1)!j}{(2i)!}
        \left(\frac{t+1}{2}\right)^i I(t).
\end{align*}
Here, 
\begin{align*}
I(t) = \int_{-1}^{1}
        (1-x^2)^{i-\frac12}
        C_{j-i}^{(i)}\left(
            \frac{t+1}{2}x+\frac{t-1}{2}
        \right)
        \md x
        =
        \frac{
            2\sqrt{\pi}\,
            \Gamma\left(i+\frac12\right)
        }{
            (i-1)!(i+j)
        }
        P_{j-i}^{(-1,2i)}(t),
\end{align*}
which is obtained by using the hypergeometric representation of the Gegenbauer polynomial~\cite[\S IV]{sze}, the beta integral, and \cref{eq:jacobi_minus_one_left}. Thus, 
\begin{align}\label{eq:restriction_analytical_expression_chebyshev}
\int_{-1}^{1}
        T_i(x)
        T_j\left(
            \frac{t+1}{2}x+\frac{t-1}{2}
        \right)
        \frac{\md x}{\sqrt{1-x^2}}
        &=
        \frac{\pi j}{i+j}
        \left(\frac{t+1}{2}\right)^i
        P_{j-i}^{(-1,2i)}(t)
\end{align}
or, equivalently, \cref{eq:restriction_analytical_expression_neg1_t}. 

Replacing $x$ and $t$ by $-x$ and $-t$, respectively, and using the Jacobi symmetry relation \cref{eq:jacobi_parity} give \cref{eq:restriction_analytical_expression_t_1}.
\end{proof}

Here and throughout, $P_n^{(-1,\beta)}$ denotes the $n$th generalized Jacobi polynomial with the first parameter $-1$, whose definition follows \cref{eq:jacobi_minus_one_zero,eq:jacobi_minus_one_left}.


\section{Spectral approximation to smooth linear Volterra integral operators}\label{sec:smooth_volt}

To construct a spectral approximation to the smooth linear Volterra operator \cref{eq:smooth_volterra} in the Chebyshev basis, we apply it to $\mathbf{T}(x)$
\begin{align}\label{eq:volterra_matrix}
\left(\mV \mathbf{T}\right)(x) = \Bigl( \mV T_0(x)\, \Big|\, \mV T_1(x) \,\Big|\, \mV T_2(x) \,\Big|\, \cdots \,\Bigr)  = \mathbf{T}(x)\mG,
\end{align}
where $\mG$ is the infinite-dimensional matrix approximation we seek.

We introduce the symmetric bi-affine mapping
\begin{align}\label{eq:symmetric_affine_map}
    \psi(x,t)
    &:=
    \frac{x+1}{2}t+\frac{x-1}{2},
    \qquad (x,t)\in[-1,1]^2,
\end{align}
which satisfies $\psi(x,t)=\psi(t,x)$, and, for each $x\in[-1,1]$, maps $[-1,1]$ onto $[-1,x]$. Thus, for a given function $u(t)$ defined on $t\in[-1,1]$,
\begin{align}\label{eq:restriction}
u\bigl(\psi(x,t)\bigr)=(\mC_{[-1,x]}u)(t).
\end{align}
With \cref{eq:symmetric_affine_map}, the affinely rescaled Duffy map~\cite{duf}, also known as a collapsed-coordinate map, is defined as
\begin{align}\label{eq:duffy}
    (x,t)
    \mapsto
    \bigl(x,\psi(x,t)\bigr),
    \qquad (x,t)\in[-1,1]^2,
\end{align}
which maps the canonical square $[-1,1]^2$ onto the triangular domain $T^2$ by collapsing the edge $x=-1$ to the vertex $(-1,-1)$. With \cref{eq:duffy}, the smooth Volterra integral operator can be rewritten as
\begin{equation}\label{eq:duffy_volterra}
(\mV u)(x)
    =
    \frac{x+1}{2}
    \int_{-1}^1
    K\bigl(
        x,\psi(x,t)
    \bigr)
    (\mC_{[-1,x]}u)(t)\,\md t.
\end{equation}
We then approximate the transformed kernel by a bivariate Chebyshev approximant
\begin{align}\label{eq:kernel_duffy_expansion_st}
K\bigl(x,\psi(x,t)\bigr)
    \approx
    \sum_{i=0}^n\sum_{j=0}^m
    C_{ij}T_i(x)T_j(t).
\end{align}
If $K(x, y)$ is smooth on $T^2$, the composite function $K\bigl(x,\psi(x,t)\bigr)$ is smooth on $[-1,1]^2$, since $\psi(x,t)$ is merely polynomial and introduces no singularity. Thus, we sample $K\bigl(x,\psi(x,t)\bigr)$ on a prescribed tensor-product Chebyshev grid and compute the coefficients $C_{ij}$ by a two-dimensional DCT. The truncation degrees $n$ and $m$ are chosen so that the bivariate Chebyshev approximant is sufficiently accurate. 

When $K\bigl(x,\psi(x,t)\bigr)$ is numerically low rank, namely, when a significant portion of the singular values of $\{C_{ij}\}$ are negligibly small, the adaptive cross approximation (ACA) usually allows for a faster construction of \cref{eq:kernel_duffy_expansion_st}. However, $K\bigl(x,\psi(x,t)\bigr)$ is often not sufficiently low-rank for ACA to be computationally advantageous. 

By substituting \cref{eq:kernel_duffy_expansion_st} into \cref{eq:duffy_volterra}, we obtain the $\ell$th column of $\left(\mV \mathbf{T}\right)(x)$
\begin{align}\label{eq:volterra_cheb_approximation}
(\mV T_\ell)(x)
    &=
    \frac{1+x}{2}
    \underbrace{
        \sum_{i=0}^n\sum_{j=0}^m
        C_{ij}T_i(x)
        \int_{-1}^1
        T_j(t)(\mC_{[-1,x]}T_\ell)(t)\,\md t
    }_{\displaystyle =\sum_{r=0}^{n+\ell}A_{r\ell}T_r(x)},
\end{align}
where $A_{r\ell}$ are the Chebyshev coefficients of the double summation. The following theorem gives a closed-form expression for $A_{r\ell}$, from which $\mG$ immediately follows.

\begin{theorem}\label{thm:volterra_matrix}
The spectral approximation to the smooth linear Volterra integral operator \cref{eq:smooth_volterra} is
\begin{align}\label{eq:volterra_matrix_chebyshev}
\mG = \mM_0\left[\frac{1+x}{2}\right]\mA,
\end{align}
where the entries of $\mA$ are
\begin{subnumcases}{
    A_{r\ell}=
    }
    \displaystyle
    2\sum_{i=0}^{n}\delta_{ri}
    \sum_{\substack{0\leqslant j\leqslant m\\ j\ \mathrm{even}}}
    \frac{C_{ij}}{1-j^2},
    &
    \(\ell=0\),
    \label{eq:A_rell_chebyshev_ell0}
    \\
    \displaystyle
    \frac{2-\delta_{r0}}{2}
    \sum_{i=0}^n\sum_{j=0}^m C_{ij}
    \left(
        \frac{\ell}{\ell+|i-r|}H_{j,|i-r|,\ell}
        +
        \frac{\ell}{\ell+i+r}H_{j,i+r,\ell}
    \right),
    &
    \(\ell\geqslant1\),
    \label{eq:A_rell_chebyshev_ellpos}
\end{subnumcases}
where
\begin{align}\label{eq:Hjkell}
        H_{j,k,\ell}
        =
        \int_{-1}^1
        \left(\frac{1+t}{2}\right)^k
        T_j(t)P_{\ell-k}^{(-1,\,2k)}(t)\,\md t.
\end{align}
For $k>\ell$, $H_{j,k,\ell} = 0$.
\end{theorem}

\begin{proof}
The expression \cref{eq:A_rell_chebyshev_ell0} follows from \cref{eq:chebyshev_integral}.
For $\ell\geqslant1$,
\begin{align*}
A_{r\ell}
    &=\frac{2-\delta_{r0}}{\pi}
    \sum_{i=0}^n\sum_{j=0}^m C_{ij}
    \int_{-1}^1 T_j(t)
    \left(
        \int_{-1}^1
        T_i(x)T_r(x)
        T_\ell\bigl(\psi(x,t)\bigr)
        \frac{\md x}{\sqrt{1-x^2}}
    \right)\md t,
\end{align*}
which is obtained by \cref{eq:chebyshev_orthogonality}. By \cref{eq:chebyshev_product}, the symmetry $\psi(x,t)=\psi(t,x)$, and \cref{eq:restriction_analytical_expression_chebyshev}, we also have
\begin{align*}
&\int_{-1}^1
    T_i(x)T_r(x)
    T_\ell\bigl(\psi(x,t)\bigr)
    \frac{\md x}{\sqrt{1-x^2}}
    \notag\\
=&
    \frac{\pi}{2}
    \left(
        \frac{\ell}{\ell+|i-r|}
        \left(\frac{1+t}{2}\right)^{|i-r|}
        P_{\ell-|i-r|}^{(-1,\,2|i-r|)}(t)
        +
        \frac{\ell}{\ell+i+r}
        \left(\frac{1+t}{2}\right)^{i+r}
        P_{\ell-i-r}^{(-1,\,2i+2r)}(t)
    \right).
\end{align*}
The last two equations lead to \cref{eq:A_rell_chebyshev_ellpos}.

It follows from the definition \cref{eq:Hjkell,eq:restriction_analytical_expression_chebyshev} and the symmetry $\psi(x,t)=\psi(t,x)$ that
\begin{align*}
H_{j,k,\ell} = \frac{k+\ell}{\pi\ell}\int_{-1}^1 T_j(t)\left( \int_{-1}^1 T_k(x)T_{\ell}\left( \frac{1+t}{2}x+\frac{t-1}{2} \right)\frac{\md x}{\sqrt{1-x^2}} \right) \md t.
\end{align*}
When $k > \ell$, the inner integral vanishes due to orthogonality. Thus, $H_{j,k,\ell} = 0$ for $k>\ell$.
\end{proof}

An immediate implication of $H_{j,k,\ell}=0$ for $k>\ell$ is that $\mA$ is lower-banded.

\subsection{Sparse structure}
Since $\mA$ is lower banded and $\mM_0\left[(1+x)/2\right]$ is tridiagonal, $\mG$ is in principle lower-banded. As we show below, $\mG$ is in fact almost-banded. This is due to the cancellations that occur in forming the product $\mM_0\left[(1+x)/2\right] \mA$. 

First, it is well known that the spectral approximation to the antiderivative operator is almost-banded \cite[\S2.4.4]{mas}. That is,
\begin{align}
\int_{-1}^{x}\mathbf{T}(y)\,\md y = \mathbf{T}(x)\mQ,  \label{eq:int_cheb}
\end{align}
where the infinite-dimensional matrix $\mQ$ can be written as $\mQ=\mB+\be\bv^{T}$. Here, $\mB$ is tridiagonal and $\be$ is the first column of the infinite identity matrix. Thus, $\be\bv^{T}$ is zero except its first row. 

Suppose that the original kernel $K(x,y)$ is a bivariate polynomial of degrees $n'$ and $m'$ in $x$ and $y$, respectively, i.e.
\begin{align}\label{eq:bivariate_monomial_series}
K(x,y) = \sum_{i=0}^{n'}\sum_{j=0}^{m'} \hat{K}_{ij}x^i y^j,
\end{align}
where $\hat{K}_{ij}$ are the monomial coefficients. In addition to \cref{eq:volterra_matrix_chebyshev}, the bivariate monomial series admits an alternative representation of $\mG$
\begin{align}\label{eq:monomial_volterra_structure}
    \mG
    =
    \sum_{i=0}^{n'}\sum_{j=0}^{m'}
    \hat{K}_{ij}\mM_0[x]^i\mQ\mM_0[x]^j
    =
    \sum_{i=0}^{n'}\sum_{j=0}^{m'}
    \hat{K}_{ij}
    \left(
        \mM_0[x]^i\mB\mM_0[x]^j
        +
        \mM_0[x]^i\be\bv^{T}\mM_0[x]^j
    \right).
\end{align}
Since $\mM_0[x]$ has bandwidths $(1,1)$, the summand $\mM_0[x]^i\mB\mM_0[x]^j$ has bandwidths $(i+j+1,i+j+1)$. On the other hand, $\mM_0[x]^i\be$ is the column vector of the Chebyshev coefficients of $x^i$ whose elements are all zero except the first $i+1$ ones. Hence, $\mG$ is an almost-banded matrix with bandwidths at most $(n'+m'+1,n'+m'+1)$ and the top $n'+1$ rows being dense.

The task therefore reduces to applying the inverse Duffy transform to $K\bigl(x,\psi(x,t)\bigr)$ in \cref{eq:kernel_duffy_expansion_st} and finding the degree of the resulting approximate kernel in the primitive variables $x$ and $y$. This approach is, however, unwieldy, if possible at all, as rational functions are inevitably involved. Hence, we choose to work the other way around. That is, we apply the Duffy transform to \cref{eq:bivariate_monomial_series} to obtain an expression for $K\bigl(x,\psi(x,t)\bigr)$ that is alternative to \cref{eq:kernel_duffy_expansion_st}:
\begin{align}\label{eq:two_expressions_kernel}
    K\bigl(x,\psi(x,t)\bigr)
    \approx
    \sum_{i=0}^{n'}\sum_{j=0}^{m'}
    \hat{K}_{ij}x^i
    \left(
        \frac{x+1}{2}t+\frac{x-1}{2}
    \right)^j = \sum_{i=0}^n\sum_{j=0}^m C_{ij}T_i(x)T_j(t)
\end{align}
It follows immediately that $n' \leqslant n$ and therefore that the number of the top dense rows is at most $n+1$. Moreover, since each nonzero term $\hat{K}_{ij}x^iy^j$ contributes a term of degree $i+j$ in $x$ after the Duffy transform, we have $i+j\leqslant n$ whenever $\hat{K}_{ij}\neq0$. Hence, the bandwidths of $\mG$ are at most $(n+1,n+1)$, provided that the Duffy compatibility condition\footnote{A bivariate polynomial $F(x,t)$ on $[-1,1]^2$ is called \emph{Duffy-compatible}, if it can be written as $F(x,t)=\widetilde{F}\bigl(x,\psi(x,t)\bigr)$ for another bivariate polynomial $\widetilde{F}(x,y)$ on $T^2$, or, equivalently, if the kernel $F(x,t)$ can be written as $\sum_j a_j(x)(t+1)^j$ where $a_j(x)$ is divisible by $(x+1)^j$. Hence, this condition is not necessarily satisfied by a random set of coefficients $\{C_{ij}\}_{i, j=0}^{n, m}$ of the bivariate Chebyshev approximant in \cref{eq:kernel_duffy_expansion_st}.} is satisfied.

The facts that $\mG$ is almost-banded and that $\mM_0[(1+x)/2]$ is tridiagonal suggest that we do not have to evaluate all the entries in $\mA$ but the ones that belong to
\begin{align}\label{eq:cheb_volterra_computational_region}
\Omega = \{(r,\ell): 0\leqslant r\leqslant n+1, \forall \ell \, \lor \, r > n+1,\ r-n \leqslant \ell \leqslant n+r+2\}.
\end{align}
Thus, for any $(r,\ell)\in\Omega$ we only need to evaluate $H_{j,k,\ell}$ with $0\leqslant j \leqslant m$ and $\max\{0, r-n\} \leqslant k \leqslant \min\{\ell,r+n\}$. This set of indices takes into account all possible values of $i+r$ and $|i-r|$ in \cref{eq:A_rell_chebyshev_ellpos} and the fact that $H_{j,k,\ell}=0$ for $k>\ell$. \Cref{fig:H-structure-a} demonstrates the moments to be evaluated by presenting $\{H_{j,k,\ell}\}_{j,k,\ell}$ as an order-3 tensor.

\begin{figure}[t!]
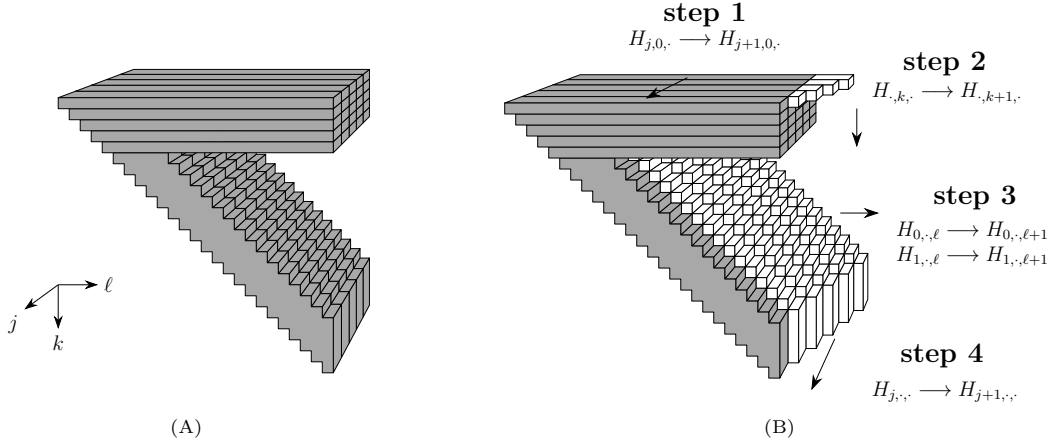

\centering
\renewcommand{\thesubfigure}{\Alph{subfigure}}
\subfloat[\label{fig:H-structure-a}]{\HRequiredMomentsPanel}%
\hspace{\HPanelGapWidth}%
\subfloat[\label{fig:H-structure-b}]{\HRecurrenceClosedPanel}
\caption{Structure of the moments $H_{j,k,\ell}$. Panel~(A): the moments, as shaded cells, used in the construction of $\mA$. Panel~(B): the shaded cells coincide with those in (A), marking the moments used by the construction of $\mA$, while the unshaded cells are the auxiliary moments that are merely intermediate byproducts of the recurrence. The four steps of recurrence are also marked with arrows indicating the recurrence directions. \label{fig:H-structure-and-recurrence}}
\end{figure}

\subsection{Fast construction by recurrence}
Direct evaluation of $\{H_{j,k,\ell}\}_{j,k,\ell}$ by quadrature is expensive. It turns out that they can be obtained by exploiting various recurrence relations. Our discussion begins with the computation of those $H_{j,k,\ell}$ contributing to the top dense rows of $\mG$, followed by those associated with the banded part.

\subsubsection{Top dense fibers}
The following theorem shows that $H_{j,k,\ell}$ for $k=0$, i.e., the top slice of the tensor depicted in \Cref{fig:H-structure-a}, can be obtained by a four- or five-term recurrence.
\begin{theorem}\label{thm:H_j0ell_recurrence}
\begin{align}\label{eq:H_j0ell_initial_data}
        H_{0,0,\ell}
        &=
        \begin{cases}
            2,  & \ell=0,\\
            -1, & \ell=1,\\
            0,  & \ell\geqslant2,
        \end{cases}
        &
        H_{j,0,0}
        &=
        \begin{cases}
            0, & j \text{ is odd},\\
            \dfrac{2}{1-j^2}, & j \text{ is even}.
        \end{cases}
\end{align}
For $\ell\geqslant1$,
\begin{align}\label{eq:H_100_recurrence}
        H_{1,0,\ell}
        =
        \frac{\ell+1}{2\ell+1}H_{0,0,\ell+1}
        -
        \frac{1}{(2\ell-1)(2\ell+1)}H_{0,0,\ell}
        +
        \frac{\ell-1}{2\ell-1}H_{0,0,\ell-1},
\end{align}
and for $j\geqslant1$ and $\ell\geqslant1$,
\begin{align}\label{eq:H_topbegin_recurrence}
        H_{j+1,0,\ell}
        =
        \frac{2(\ell+1)}{2\ell+1}H_{j,0,\ell+1}
        -
        \frac{2}{(2\ell-1)(2\ell+1)}H_{j,0,\ell}
        +
        \frac{2(\ell-1)}{2\ell-1}H_{j,0,\ell-1}
        -
        H_{j-1,0,\ell}.
\end{align}
\end{theorem}

\begin{proof}
Since $H_{0,0,\ell}=\int_{-1}^{1}P_\ell^{(-1,0)}(t)\,\md t$ by definition, $H_{0,0,0}=2$ follows from \cref{eq:jacobi_minus_one_zero}. The relation \cref{eq:jacobi_minus_one_left} with $\beta=0$ and the orthogonality of Jacobi polynomials lead to $H_{0,0,1}=-1$ and $H_{0,0,\ell}=0$ for $\ell\geqslant2$.

Since $P_0^{(-1,0)}(t)=1$, $H_{j,0,0}$ reduces to $\int_{-1}^{1}T_j(t)\,\md t$. The values of $H_{j,0,0}$ given by \cref{eq:H_j0ell_initial_data} follow from \cref{eq:chebyshev_integral}.

For $\ell\geqslant1$, combining \cref{eq:jacobi_multiplication} with the parameters set to $(-1,0)$ and the definition of $H_{1,0,\ell}$ gives \cref{eq:H_100_recurrence}. When $j\geqslant 1$, substituting \cref{eq:chebyshev_three_term} into the definition of $H_{j,0,\ell}$ and using \cref{eq:jacobi_multiplication} again yield \cref{eq:H_topbegin_recurrence}.
\end{proof}

Thus, the slice of $k=0$ is obtained as follows. We first evaluate the entries in the fibers of $j = 0$ and $\ell = 0$ using \cref{eq:H_j0ell_initial_data}. The fibers corresponding to $j = 1$ and subsequent values of $j$ are computed via recurrence relations \cref{eq:H_100_recurrence,eq:H_topbegin_recurrence}, respectively. Suppose the largest index $\ell$ we need to calculate $H_{j,0,\ell}$ for is $l$. To compute $H_{m,0,l}$ recursively, we have to have $\{H_{j,0,\ell}\}_{\ell = l+1}^{l+m-j}$ of the $j$th fiber ($j = 0, 1, \cdots, m-1$) also evaluated, though they are not required in forming $\mA$. This is due to the indices of dependence in \cref{eq:H_100_recurrence,eq:H_topbegin_recurrence}. In \Cref{fig:H-structure-b}, these auxiliary entries are indicated by the unshaded cells in the top slice.

The next theorem shows that once the initial slice is available, the remaining entries in the top dense fibers can be obtained slicewise via recurrence as $k$ increases. 

\begin{theorem}\label{thm:H_kell_recurrence}
For $k\geqslant0$ and $\ell\geqslant k+1$, 
\begin{align}\label{eq:H_top_recurrence}
        H_{j,k+1,\ell}
        =
        \frac{
        (\ell-k)H_{j,k,\ell}
        +(\ell+k)H_{j,k,\ell-1}
        -(\ell-k-2)H_{j,k+1,\ell-1}
        }{
        \ell+k
        }.
\end{align}
\end{theorem}

\begin{proof}
We first consider the boundary case $\ell=k+1$. By  \cref{eq:jacobi_minus_one_zero,eq:jacobi_minus_one_left} with $n=1$ and $\beta=2k$, 
it can be shown that
\begin{align*}
\frac{1+t}{2}P_0^{(-1, 2k+1)}(t) = \frac{1}{2k+1} \bigl(P_1^{(-1, 2k)}(t)+(2k+1)P_0^{(-1, 2k)}(t)\bigr).
\end{align*}
Multiplying both sides by $T_j(t)\bigl((1+t)/2\bigr)^k$, integrating over $[-1,1]$, using the fact that $H_{j,k+1,k}=0$ give \cref{eq:H_top_recurrence}. 

For $\ell\geqslant k+2$, we specialize \cref{eq:jacobi_parameter_shift_left} to $\beta=2k+1$ and $n=\ell-k-1,\ell-k-2$, and \cref{eq:jacobi_parameter_shift_right} to $\beta=2k$ and $n=\ell-k,\ell-k-1$. Combining the resulting identities gives
\begin{align*}
    (\ell+k)\frac{1+t}{2}
    P_{\ell-k-1}^{(-1,2k+2)}(t)
    +
    (\ell-k-2)\frac{1+t}{2}
    P_{\ell-k-2}^{(-1,2k+2)}(t)
    =
    (\ell-k)P_{\ell-k}^{(-1,2k)}(t)
    +
    (\ell+k)P_{\ell-k-1}^{(-1,2k)}(t).
\end{align*}
Multiplying this identity by $T_j(t)\bigl((1+t)/2\bigr)^k$, integrating over $[-1,1]$, and using \cref{eq:Hjkell} give \cref{eq:H_top_recurrence}.
\end{proof}

\subsubsection{Banded part}
In principle, \cref{eq:H_top_recurrence} can also be used to recurse for the entries in the banded part of each $(k,\ell)$ slice---since the entries in the main diagonal and first superdiagonal can be calculated explicitly, recurrence can be carried out diagonal-wise using a variation of \cref{eq:H_top_recurrence}
\begin{align}\label{eq:H_top_recurrence_variation}
H_{j,k-1,\ell+1}
    =
    \frac{
    (\ell+k)H_{j,k,\ell+1}
    +(\ell-k)H_{j,k,\ell}
    -(\ell+k)H_{j,k-1,\ell}
    }{\ell-k+2}.
\end{align}
Unfortunately, the recurrence \cref{eq:H_top_recurrence_variation} is numerically unstable due to severe cancellations. To this end, we instead construct a recurrence relation in the $(j,\ell)$ slice in \cref{thm:H_bottom_recurrence}, while the following theorem gives the first two fibers, i.e., $j = 0, 1$, with which the recurrence can be started.
\begin{theorem}\label{thm:H_initial_layers}
$H_{0,k,k}=2/(k+1)$, $H_{0,k,k+1}=-2(2k+1)/\bigl((k+1)(k+2)\bigr)$, and for $\ell\geqslant k+2$,
\begin{subequations}
\begin{align}\label{eq:H_0kl}
H_{0,k,\ell}
        =
        -\frac{\ell+k}{\ell+k-1}
        \frac{\ell-2}{\ell+1}
        H_{0,k,\ell-1}.
\end{align}
$H_{1,k,k}=2k/\bigl((k+1)(k+2)\bigr)$, $H_{1,k,k+1}=2(2k+1)(1-k)/\bigl((k+1)(k+2)(k+3)\bigr)$, and for $\ell\geqslant k+2$,
\begin{align}\label{eq:H_1kl}
        H_{1,k,\ell}
        =
        -\frac{\ell+k}{\ell+k-1}
        \frac{\ell-3}{\ell+2}
        \left(
            H_{1,k,\ell-1}
            +
            H_{0,k,\ell-1}
        \right)
        -
        H_{0,k,\ell}.
\end{align}
\end{subequations}
\end{theorem}
\begin{proof}
Straightforward integration gives $H_{0,k,k}$ and $H_{1,k,k}$. 

For $\ell\geqslant k+1$, the change of variable $s=(1+t)/2$ gives
\begin{align*}
    H_{0,k,\ell}
    =
    2\int_0^1
    s^k P_{\ell-k}^{(-1,2k)}(2s-1)\,\md s,
    \quad
    H_{1,k,\ell}+H_{0,k,\ell}
    =
    4\int_0^1
    s^{k+1}P_{\ell-k}^{(-1,2k)}(2s-1)\,\md s.
\end{align*}
We consider $H_{1,k,\ell}+H_{0,k,\ell}$ rather than $H_{1,k,\ell}$ because the former bears the same form as $H_{0,k,\ell}$ except the incremental power of $s$. Identity \cref{eq:jacobi_parity}, the hypergeometric representation of the Jacobi polynomials \cite[\S IV]{sze}, and termwise evaluation using the beta integral yield
\begin{align*}
H_{0,k,\ell}
    =
    \frac{2(-1)^{\ell-k}(\ell+k) (k)_{\ell-k-1}}
    {(k+1)(k+2)(k+3)_{\ell-k-1}},
    \quad
    H_{1,k,\ell}+H_{0,k,\ell}
    =
    \frac{4(-1)^{\ell-k}(\ell+k)(k-1)_{\ell-k-1}}
    {(k+2)(k+3)(k+4)_{\ell-k-1}}.
\end{align*}

When $\ell=k+1$, $H_{0,k,\ell}$ reduces to $H_{0,k,k+1}=-2(2k+1)/\bigl((k+1)(k+2)\bigr)$, while subtracting this value from $H_{1,k,\ell}+H_{0,k,\ell}$ gives $H_{1,k,k+1}=2(2k+1)(1-k)/\bigl((k+1)(k+2)(k+3)\bigr)$.

For $\ell\geqslant k+2$, $(\cdot)_{\ell-k-1}=(\cdot+\ell-k-2)(\cdot)_{\ell-k-2}$ leads to \cref{eq:H_0kl} and the recurrence relation
\begin{align*}
H_{1,k,\ell}+H_{0,k,\ell}
    =
    -\frac{(\ell+k)(\ell-3)}{(\ell+k-1)(\ell+2)}
    \left(
        H_{1,k,\ell-1}
        +
        H_{0,k,\ell-1}
    \right),
\end{align*}
which, upon subtracting $H_{0,k,\ell}$ on both sides, gives \cref{eq:H_1kl}.
\end{proof}

\begin{theorem}\label{thm:H_bottom_recurrence}
For $j\geqslant1$ and $k\geqslant0$, the diagonal entries for incremental $j$ can be recursed by
\begin{subequations}\label{eq:H_band_recurrence}
\begin{align}\label{eq:H_bottom_recurrence_diagonal}
        H_{j+1,k,k}
        =
        \frac{4}{2k+1}
        H_{j,k,k+1}
        +
        2H_{j,k,k}
        -
        H_{j-1,k,k}.
\end{align}
For $\ell\geqslant k+1$, the recurrence is
\begin{align}\label{eq:H_bottom_recurrence}
H_{j+1,k,\ell}
       = 2\frac{\ell+k}{\ell}\left(
        \frac{\ell-k+1}
        {2\ell+1}
        H_{j,k,\ell+1}
                +
        \frac{\ell-k-1}
        {2\ell-1}
        H_{j,k,\ell-1}\right)
        +2
        \frac{4k^2-1}
        {4\ell^2-1}
        H_{j,k,\ell}
        -
        H_{j-1,k,\ell}.
\end{align}
\end{subequations}
\end{theorem}

\begin{proof}
For $\ell=k$, combining \cref{eq:chebyshev_three_term,eq:Hjkell} and using the fact \cref{eq:jacobi_minus_one_zero} gives \cref{eq:H_bottom_recurrence_diagonal}.

For $\ell\geqslant k+1$, using \cref{eq:chebyshev_three_term} we have
\begin{align}\label{eq:H_bottom_recurrence_intermediate}
H_{j+1,k,\ell}+H_{j-1,k,\ell} = 2\int_{-1}^1 t \left(\frac{1+t}{2}\right)^k T_j(t)P_{\ell-k}^{(-1,2k)}(t)\,\md t.
\end{align}
Specializing \cref{eq:jacobi_multiplication} to degree $\ell-k$ and parameters $(-1,2k)$ and substituting the resulting recurrence relation into \cref{eq:H_bottom_recurrence_intermediate} yields \cref{eq:H_bottom_recurrence}.
\end{proof}

As for the top dense fibers, in any $(j, \ell)$ slice we have to evaluate more entries than we actually need for forming $A_{r\ell}$ due to the indices of dependence in \cref{eq:H_band_recurrence}---we compute $m-j$ more auxiliary entries in the $j$th fiber, as demonstrated by \Cref{fig:H-structure-b}. 

\subsubsection{Assembling $\mA$}
Once all the required entries of $H_{j,k,\ell}$ are available, we assemble $\mA$ and form $\mG$. Computing $A_{r\ell}$ for $\ell \geqslant 1$ by directly applying \cref{eq:A_rell_chebyshev_ellpos} would incur a cost that is much higher than necessary. As suggested by \cref{eq:A_rell_chebyshev_ellpos}, the $\ell$th column of $\mA$ for $\ell\geqslant 1$ can be written as
\begin{align}\label{eq:A_toeplitz_hankel}
    \mA_{:,\ell}
    =
    W\sum_{j=0}^{m}F^{j,\ell}C_{:,j},
\end{align}
where $W = \operatorname{diag}\left(1/2,1,1,\ldots  \right)$,
\begin{align}\label{eq:toeplitz_hankel}
    F^{j,\ell}
    &= 
    \underbrace{
    \begin{pmatrix}
        F_{j,0,\ell}
        & F_{j,1,\ell}
        & F_{j,2,\ell}
        & \cdots
        & F_{j,n,\ell}
        \\
        F_{j,1,\ell}
        & F_{j,0,\ell}
        & F_{j,1,\ell}
        & \ddots
        & F_{j,n-1,\ell}
        \\
        F_{j,2,\ell}
        & F_{j,1,\ell}
        & F_{j,0,\ell}
        & \ddots
        & F_{j,n-2,\ell}
        \\
        \vdots
        & \ddots
        & \ddots
        & \ddots
        & \vdots
    \end{pmatrix}}_{\text{Toeplitz}}
    +
    \underbrace{
    \begin{pmatrix}
        F_{j,0,\ell}
        & F_{j,1,\ell}
        & F_{j,2,\ell}
        & \cdots
        & F_{j,n,\ell}
        \\
        F_{j,1,\ell}
        & F_{j,2,\ell}
        & F_{j,3,\ell}
        & \iddots
        & F_{j,n+1,\ell}
        \\
        F_{j,2,\ell}
        & F_{j,3,\ell}
        & F_{j,4,\ell}
        & \iddots
        & F_{j,n+2,\ell}
        \\
        \vdots
        & \iddots
        & \iddots
        & \iddots
        & \vdots
    \end{pmatrix}}_{\text{Hankel}} ,
\end{align}
and
\begin{align*}
F_{j,k,\ell} = \frac{\ell}{\ell+k}H_{j,k,\ell}.
\end{align*}
Thus, each product $F^{j,\ell}C_{:,j}$ is the sum of a Toeplitz and a Hankel matrix-vector product, both of which can be evaluated efficiently by FFTs. It is indicated by \cref{eq:cheb_volterra_computational_region} that when $\ell \geqslant 2n+5$ the indices for the desired entries of $\mA_{:,\ell}$ are pinched into two segments and the gap between these two segments widens as $\ell$ increases. In such scenarios, only the products of $C_{:,j}$ with two submatrices of $F^{j,\ell}$ can be computed via FFTs to save even more computation. 
 
The complete algorithm for constructing $\mG$ is summarized in \Cref{alg:cheb_volterra_matrix}. To give the asymptotic complexities, we assume that the $N \times N$ principal finite section of $\mG$ is constructed. The overall cost is $\mO(mn\log n N)$ flops.

\begin{algorithm}[t]
\caption{Construction of spectral approximation to smooth linear Volterra operator}
\label{alg:cheb_volterra_matrix}
\begin{algorithmic}[1]

\REQUIRE
Kernel $K(x,y)$ and truncation size $N$.

\ENSURE
The finite truncation $\mG$.

\STATE Compute the Chebyshev coefficients $C$ in
\cref{eq:kernel_duffy_expansion_st}.
\hfill
$\triangleright\ \mO\!\bigl(nm(\log n+\log m)\bigr)$

\STATE Compute the top dense fibers of $H_{j,k,\ell}$ by following
\cref{thm:H_j0ell_recurrence,thm:H_kell_recurrence}.
\hfill
$\triangleright\ \mO(mnN)$

\STATE Compute the required banded part of $H_{j,k,\ell}$ by following
\cref{thm:H_initial_layers,thm:H_bottom_recurrence}.
\hfill
$\triangleright\ \mO\!\bigl(m(n+m)N\bigr)$

\STATE Assemble $\mA$ using
\cref{eq:A_rell_chebyshev_ell0,eq:A_toeplitz_hankel}.
\hfill
$\triangleright\ \mO(mn\log n\,N)$

\STATE Form $\mG$ by
\cref{eq:volterra_matrix_chebyshev}.
\hfill
$\triangleright\ \mO(nN)$

\end{algorithmic}
\end{algorithm}

For $N \gg n$, since the $N \times N$ truncation of $\mG$ has $n+1$ dense leading rows and an $(n+1, n+1)$ band, both the matrix construction and linear solve via, e.g., adaptive QR \cite{olv}, are linear in $N$. 

Finally, we remark that the spectral approximation to smooth Volterra operators obtained in the same way but with Legendre polynomials is strictly banded. However, the cost of construction increases to $\mO(mn^2N)$ due to the absence of the Toeplitz-plus-Hankel structure \cref{eq:toeplitz_hankel}.


\section{Evaluation of nonlinear Volterra integrals with smooth kernels}\label{sec:nonlinear_volt}
The nonlinear Volterra integral operator \cref{eq:nonlinear_volterra} is the key building block of nonlinear Volterra equations, which are typically solved by iterative methods such as Picard or Newton methods. In turn, such methods require repeated evaluation of the nonlinear Volterra integral operator or its Fr\'{e}chet derivative, so their overall efficiency depends strongly on that of this evaluation.

Suppose that the current approximation to the solution is $u(x)= \sum_{j=0}^{N}u_jT_j(x)$, which is being refined by the Picard iteration. A most common approach works on an $(N+1)$-point Chebyshev grid and for each $0\leqslant i \leqslant N$ evaluates $\int_{-1}^{x_i}K\bigl(x_i,y,u(y)\bigr)\,\md y$ using the Clenshaw algorithm for evaluating $u(y)$ at the Chebyshev points of $[-1, x_i]$, followed by the Clenshaw--Curtis quadrature. Since this is done for all $\{x_i\}_{i=0}^N$, the total complexity is $\mO(N^3)$.

We take a different approach for a cheaper evaluation of \cref{eq:nonlinear_volterra} by relating the evaluation of $u(y)$ for different $x_i$ with the restriction operator. In fact, applying the bi-affine transform \cref{eq:symmetric_affine_map}, we have
\begin{align}
\int_{-1}^{x}K\bigl(x,y,u(y)\bigr)\,\md y
    =
    \frac{1+x}{2}
    \int_{-1}^{1}
    K\Bigl(x,\psi(x,t),u\bigl(\psi(x,t)\bigr)\Bigr)\,\md t .
\end{align}
In \cref{eq:duffy_volterra} or \cref{eq:volterra_cheb_approximation}, since the bivariate restricted function $u\bigl(\psi(x,t)\bigr)$, as part of the integrand, is integrated with respect to $t$, we only care to know the result as a univariate function in $x$. Here, the situation is different because the nonlinearity prevents us from immediately getting rid of the variable $t$, so $u\bigl(\psi(x,t)\bigr)$ must be treated as a bivariate function. Our approach is to construct a bivariate Chebyshev approximant to $u\bigl(\psi(x,t)\bigr)$ using a differential recurrence relation (see \cref{eq:nonlinear_dk_differential_recurrence} below), followed by evaluating this approximant on a Chebyshev grid by two-dimensional DCTs. This is key to reducing the computational cost.

\subsection{Bivariate Chebyshev approximant of $u\bigl(\psi(x,t)\bigr)$}
By \cref{thm:restriction_analytical_expression}, we have
\begin{align}\label{eq:restricted_u_expansion}
u\bigl(\psi(x,t)\bigr)
    =
    \sum_{k=0}^{N}\hat{C}_k(x)T_k(t) = \sum_{k=0}^{N}\sum_{i=0}^{N}\hat{C}_{ik}T_i(x)T_k(t),
\end{align}
where
\begin{align}\label{eq:restricted_u_coefficients}
    \hat{C}_k(x)
    =
    \frac{2-\delta_{k0}}{2}
    \left(\frac{1+x}{2}\right)^k
    \left(
        \delta_{k0}u_0
        +
        \sum_{\ell=0}^{N-k}
        u_{k+\ell}P_\ell^{(0,2k)}(x)
        -
        \sum_{\ell=0}^{N-k-1}
        u_{k+\ell+1}P_\ell^{(0,2k)}(x)
    \right)
\end{align}
and $\hat{C}_{ik}$ are the Chebyshev coefficients of $\hat{C}_k(x)$. Ideally, $\hat{C}_{ik}$ can be obtained from evaluating the closed-form expression \cref{eq:restricted_u_coefficients} on the $(N+1)$-point Chebyshev grid and a value-to-coefficient transform. This is, however, hardly practical---evaluation of Jacobi polynomials with large parameters is subject to over and underflow and, more importantly, loss of relative accuracy. In addition, the evaluation of \cref{eq:restricted_u_coefficients} using the Clenshaw algorithm would incur a total cost of $\mO(N^3)$ flops. To this end, we compute $\hat{C}_{ik}$ using the recurrence relation satisfied by $\hat{C}_k(x)$, which is given by the following theorem.
\begin{theorem}\label{thm:nonlinear_dk_differential_recurrence}
For $0\leqslant k\leqslant N-1$,
\begin{align}\label{eq:nonlinear_dk_differential_recurrence}
\frac{2-\delta_{k0}}{2}
        \bigl(
            (1+x)\frac{\md}{\md x}+k+1
        \bigr)\hat{C}_{k+1}(x) = 
        \bigl(
            (1+x)\frac{\md}{\md x}-k
        \bigr)\hat{C}_k(x).
\end{align}
\end{theorem}

\begin{proof}
This recurrence relation can be verified by substituting \cref{eq:restricted_u_coefficients} into \cref{eq:nonlinear_dk_differential_recurrence}, applying \cref{eq:jacobi_derivative} to the right-hand side and \cref{eq:jacobi_differential_shift} specialized to $\beta=2k+2$ to the left-hand side.
\end{proof}

Given $\hat{C}_k(x)$, we can solve \cref{eq:nonlinear_dk_differential_recurrence} for $\hat{C}_{k+1}(x)$ without enforcing a boundary condition. To see this, let $\mathbb{P}_N$ be the space of polynomials of degree at most $N$. For $k\geqslant0$,
\begin{align*}
    \bigl((1+x)\frac{\md}{\md x}+k+1\bigr)(1+x)^j
    =
    (j+k+1)(1+x)^j.
\end{align*}
In other words, the operator $(1+x)\md/\md x+k+1$ maps $\mathbb{P}_N$ into itself and its matrix representation in the monomial basis $\{(1+x)^j\}_{j=0}^{N}$ is therefore diagonal with nonzero diagonal entries. Hence, the operator is invertible on $\mathbb{P}_N$, and \cref{eq:nonlinear_dk_differential_recurrence} alone uniquely determines $\hat{C}_{k+1}(x)$ without the need of a boundary condition.

The differential recurrence \cref{eq:nonlinear_dk_differential_recurrence} can be solved using the ultraspherical spectral method~\cite{olv}, which leads to the finite-dimensional system
\begin{align}\label{eq:nonlinear_dk_coefficient_recurrence}
    \frac{2-\delta_{k0}}{2}
    \bigl(
        M_1[1+x]D_1
        +
        (k+1)S_0
    \bigr)\hat{C}_{:,k+1}
    =
    \bigl(
        M_1[1+x]D_1
        -
        kS_0
    \bigr)\hat{C}_{:,k}.
\end{align}
Here, $M_1[1+x]$, $D_1$, and $S_0$ denote the $(N+1)\times(N+1)$ finite truncations of $\mM_1[1+x]$, $\mD_1$, and $\mS_0$, respectively. The matrices on both sides have bandwidth $(0,2)$. Consequently, each recurrence step requires $\mO(N)$ operations, and all the coefficients $\hat{C}_{:,:}$ can be obtained at a cost of $\mO(N^2)$ operations.

To start the recurrence, we need the Chebyshev coefficients of $\hat{C}_0(x)$, which, as given by \cref{eq:restricted_u_coefficients}, is a Legendre series. Its Chebyshev coefficients can be obtained either by evaluating this Legendre series on a Chebyshev grid using the Clenshaw algorithm followed by a value-to-coefficient transform, or directly by a fast Legendre-to-Chebyshev transform.

Since $\psi(x,t)=\psi(t,x)$, the function $u\bigl(\psi(x,t)\bigr)$ is symmetric in $x$ and $t$ and so is its Chebyshev coefficient matrix $\hat{C}=(\hat{C}_{ij})_{i,j=0}^{N}$, i.e., $\hat{C}_{ij}=\hat{C}_{ji}$. This symmetry is exploited to reduce the computational cost by a factor of 2.

\subsection{Numerical quadrature}
After the coefficient matrix $\hat{C}$ of $u\bigl(\psi(x,t)\bigr)$ has been constructed, its values on a tensor-product Chebyshev grid can be computed efficiently by two-dimensional DCTs. Since the degree for adequate resolution of the composite integrand $K\Bigl(x,\psi(x,t),u\bigl(\psi(x,t)\bigr)\Bigr)$ may be larger than $N$, we use an independent degree $M \geqslant N$. Let $\{x_i^{(M)}\}_{i=0}^M$ and $\{t_j^{(M)}\}_{j=0}^M$ be the Chebyshev grids on $[-1,1]$. The coefficient-to-value transform via two-dimensional DCTs gives
\begin{align}
U_{ij}^{(M)}
    =
    u\bigl(\psi(x_i^{(M)},t_j^{(M)})\bigr).
\end{align}
The nonlinear Volterra integral is then evaluated using the Clenshaw--Curtis quadrature
\begin{align}\label{eq:nonlinear_volterra_quadrature}
    v_i^{(M)}
    =
    \frac{1+x_i^{(M)}}{2}
    \sum_{j=0}^{M}w_j^{(M)}
    K\bigl(
        x_i^{(M)},
        \psi(x_i^{(M)},t_j^{(M)}),
        U_{ij}^{(M)}
    \bigr),
\end{align}
where $w_j^{(M)}$ are the Clenshaw--Curtis weights.
Finally, we obtain the Chebyshev coefficients $\{\hat{v}_i^{(M)}\}_{i =0}^M$ from $\{v_i^{(M)}\}_{i =0}^M$ using DCT.

The degree $M$ of the integrand can be selected adaptively. Starting from $M=N$, we evaluate the integral on grids of degrees $M$ and $2M$. Since the Chebyshev grids are nested, $x_{2j}^{(2M)}=x_j^{(M)}$ for $0\leqslant j\leqslant M$, the two approximations can be compared at their common nodes. Let $\bv^{(2M)}$ denote the value vector of nonlinear Volterra integral evaluated on the finer grid, and define
\begin{align}\label{eq:nonlinear_adaptive_indicators}
    \eta
    =
    \frac{
        \max\limits_{0\leqslant j\leqslant M}
        \left|v_j^{(M)}-v_{2j}^{(2M)}\right|
    }{
        \max\left\{1,\left\|\bv^{(2M)}\right\|_\infty\right\}
    },
\end{align}
which measures the difference between the two grids. The degree-$2M$ result is accepted when $\eta$ is below a prescribed tolerance $\epsilon$ and a coefficient plateau of sufficiently small magnitude is detected. Otherwise, $M$ is doubled and the process is repeated. The tolerance $\epsilon$ is usually chosen as a moderate multiple of machine precision. 

The complete procedure is summarized in \Cref{alg:nonlinear_volterra_evaluator}. Constructing the coefficient matrix $\hat{C}$ requires $\mO(N^2)$ operations, and the two-dimensional DCTs require $\mO(M^2\log M)$ operations, whereas pointwise kernel evaluation and row-wise quadrature require $\mO(M^2)$ operations, assuming constant-cost kernel evaluations. Therefore, the total complexity is $\mO(M^2\log M)$, and in particular $\mO(N^2\log N)$ when $M=\mO(N)$.

\begin{algorithm}[t] 
\caption{Adaptive evaluation of smooth nonlinear Volterra integral operator} 
\label{alg:nonlinear_volterra_evaluator} 
\begin{algorithmic}[1] 
 
\REQUIRE 
Chebyshev coefficient vector $\bu$ of the degree-$N$ approximant to the solution; kernel $K(x,y,z)$; quadrature tolerance $\epsilon$.
 
\ENSURE 
Chebyshev coefficients of \cref{eq:nonlinear_volterra}. 
 
\STATE Compute $\hat{C}_{:,0}$ from 
\cref{eq:restricted_u_coefficients}. 
\hfill $\triangleright\ \mO(N^2)$ 
 
\STATE Recurse for $\hat{C}_{:,1},\ldots,\hat{C}_{:,N}$ using 
\cref{eq:nonlinear_dk_coefficient_recurrence}. 
\hfill $\triangleright\ \mO(N^2)$ 
 
\STATE Set $M\gets N/2$, $\eta\gets\infty$, and 
$\mathrm{plateau}\gets\mathrm{false}$. 
 
\WHILE{$\eta>\epsilon\ \vee\ \neg\mathrm{plateau}$} 
    \STATE Set $M\gets2M$ 
 
    \STATE Compute $\{v_i^{(M)}\}_{i=0}^M$ and $\{v_i^{(2M)}\}_{i=0}^{2M}$ following \cref{eq:nonlinear_volterra_quadrature}. 
    \hfill $\triangleright\ \mO(M^2\log M)$ 
 
    \STATE Compute $\{\hat{v}_i^{(2M)}\}_{i=0}^{2M}$ from $\{v_i^{(2M)}\}_{i=0}^{2M}$ using DCT. 
    \hfill $\triangleright\ \mO(M\log M)$
 
    \STATE Compute $\eta$ following \cref{eq:nonlinear_adaptive_indicators} and check the coefficient plateau of $\{\hat{v}_i^{(2M)}\}_{i=0}^{2M}$. 
 
\ENDWHILE 
 
\STATE Truncate $\{\hat{v}_i^{(2M)}\}_{i=0}^{2M}$ according to $\epsilon$.
 
\end{algorithmic} 
\end{algorithm}

Notably, the adaptivity of $M$ ensures the resolution of the integrand and, therefore, the quadrature, whereas the solution degree $N$ is usually dictated by the nonlinear solver. In a Picard iteration, $N$ is increased when the Chebyshev coefficients of the approximate solution do not exhibit a plateau, i.e., when the current iterate is not yet sufficiently resolved. In such a case, the current iterate is elongated by zero-padding for the next iteration. Once a coefficient plateau is detected, the nonlinear residual is checked to gauge the convergence of the Picard iteration.

\section{Spectral approximation to weakly singular Volterra integral operators}\label{sec:weaksingular_volt}
To construct the spectral approximation to the weakly singular Volterra operator \cref{eq:singular_volterra}, we use the Chebyshev fractional polynomials (CFPs)~\cite{pu}
\begin{align}\label{eq:cfp_definition}
    Q_n^{\sigma,\gamma}(x)
    =
    \left(\frac{1+x}{2}\right)^\sigma
    T_n\bigl(\varphi(x)\bigr),
\end{align}
where
\begin{align}\label{eq:phi}
\varphi(x)
    =
    2\left(\frac{1+x}{2}\right)^\gamma-1.
\end{align}
Here, $\gamma>0$ is chosen such that $1-\mu=p\gamma$ for some $p\in\mathbb{N}_+$. This ensures that $\bigl((1+x)/2\bigr)^{1-\mu}$ has a finite expansion in the $Q_n^{0,\gamma}$ basis and, consequently, that the singular Volterra integral can also be represented in the same CFP basis. The other parameter $\sigma>-1$ is chosen so that the weight $\bigl((1+x)/2\bigr)^{\sigma}$ captures the singular behavior of the problem being considered. The CFPs are orthogonal with respect to the weight function
\begin{align}\label{eq:cfp_orthogonality}
    w_{\sigma,\gamma}(x)
    =
    \left(\frac{1+x}{2}\right)^{-2\sigma}\hspace{-3mm}
    \frac{\varphi'(x)}
    {\sqrt{1-\varphi(x)^2}},
\end{align}
in the sense that
\begin{align}\label{eq:cfp_orthogonality_relation}
    \int_{-1}^{1}
    Q_i^{\sigma,\gamma}(x)
    Q_r^{\sigma,\gamma}(x)
    w_{\sigma,\gamma}(x)\,\md x
    =
    \frac{\pi}{2-\delta_{r0}}\delta_{ri}.
\end{align}

When $K(x,y)$ is constant, \cref{eq:singular_volterra} reduces to a fractional integral operator, whose spectral approximation in the CFP basis is studied in~\cite{liu}. For general $K(x,y)$, we again apply the Duffy transformation to obtain
\begin{align}\label{eq:weak_duffy_transform_integral}
    (\mV u)(x)
    =
    \left(\frac{1+x}{2}\right)^{1-\mu}
    \int_{-1}^{1}
    (1-t)^{-\mu}
    K\bigl(x,\psi(x,t)\bigr)
    (\mC_{[-1,x]}u)(t)
    \,\md t.
\end{align}
As in the smooth linear case, the Duffy-transformed kernel is approximated by a finite bivariate CFP expansion
\begin{align}\label{eq:weak_kernel_cfp_expansion}
     K\bigl(x,\psi(x,t)\bigr) \approx 
    \sum_{i=0}^{n}\sum_{j=0}^{m}
    \tilde{C}_{ij}Q_i^{0,\gamma}(x)Q_j^{0,\gamma}(t),
\end{align}
where the truncation degrees $n$ and $m$ are chosen to achieve the desired accuracy in approximating $K\bigl(x,\psi(x,t)\bigr)$. Since \cref{eq:weak_kernel_cfp_expansion} is equivalent to a bivariate Chebyshev approximation of $K\Bigl(\varphi^{-1}(\xi), \psi\bigl(\varphi^{-1}(\xi),\varphi^{-1}(\zeta)\bigr)\Bigr)$ in the variables $\xi$ and $\zeta$, the convergence of the finite bivariate CFP expansion is governed by the regularity of this triply composite function. In particular, if $1/\gamma\in\mathbb{N}_+$, then $\varphi^{-1}(\cdot)$ is a polynomial, and the composite function is smooth whenever $K$ is smooth. In this case, the CFP coefficients $\tilde{C}_{ij}$ can be computed using a two-dimensional DCT. 

As in the smooth linear case, we apply the weakly singular Volterra operator to the CFP quasimatrix $\mathbf{Q}^{\sigma,\gamma}(x) = \Bigl(
Q_0^{\sigma,\gamma}(x)
\,\Big|\,
Q_1^{\sigma,\gamma}(x)
\,\Big|\,
Q_2^{\sigma,\gamma}(x)
\,\Big|\,
\cdots \Bigr)$, which gives
\begin{align}\label{eq:weak_volterra_matrix}
\mV\mathbf{Q}^{\sigma,\gamma}(x)
    =
    \Bigl(
        \mV Q_0^{\sigma,\gamma}(x)
        \,\Big|\,
        \mV Q_1^{\sigma,\gamma}(x)
        \,\Big|\,
        \mV Q_2^{\sigma,\gamma}(x)
        \,\Big|\,
        \cdots
    \Bigr)
    =
    \mathbf{Q}^{\sigma,\gamma}(x)\tilde{\mG},
\end{align}
where $\tilde{\mG}$ is the infinite-dimensional spectral approximation to $\mV$ that we seek. To evaluate $\mV Q_{\ell}^{\sigma,\gamma}(x)$ for $\ell \geqslant 0$, we need $(\mC_{[-1,x]}Q_{\ell}^{\sigma,\gamma})(t)$. It can be verified that for $\ell\geqslant0$
\begin{align*}
    \frac{1+\psi(x,t)}{2}
    =
    \frac{(1+x)(1+t)}{4}~~
    \text{ and }~~
    \varphi\bigl(\psi(x,t)\bigr)
    =
    \psi\bigl(\varphi(t),\varphi(x)\bigr),
\end{align*}
which imply 
\begin{align}\label{eq:restrictQ}
(\mC_{[-1,x]}Q_{\ell}^{\sigma,\gamma})(t) = 
    Q_\ell^{\sigma,\gamma}\bigl(\psi(x,t)\bigr)
    =
    \left(\frac{1+x}{2}\right)^\sigma
    \left(\frac{1+t}{2}\right)^\sigma
    \bigl(
        \mC_{[-1,\varphi(t)]}T_\ell
    \bigr)\bigl(\varphi(x)\bigr).
\end{align}
Substituting \cref{eq:weak_kernel_cfp_expansion,eq:restrictQ} into \cref{eq:weak_duffy_transform_integral} yields the $\ell$th column of $\mV\mathbf{Q}^{\sigma,\gamma}(x)$
\begin{align}\label{eq:weak_volterra_basis_action}
    (\mV Q_\ell^{\sigma,\gamma})(x)
    =
    \left(\frac{1+x}{2}\right)^{1-\mu}
    \underbrace{
        \sum_{i=0}^{n}\sum_{j=0}^{m}
        \tilde{C}_{ij}Q_i^{\sigma,\gamma}(x)
        \int_{-1}^{1}
        (1-t)^{-\mu}
        Q_j^{\sigma,\gamma}(t)
        \bigl(
            \mC_{[-1,\varphi(t)]}T_\ell
        \bigr)\bigl(\varphi(x)\bigr)
        \,\md t
    }_{\displaystyle
        =
        \sum_{r=0}^{n+\ell}
        \tilde{A}_{r\ell}Q_r^{\sigma,\gamma}(x)
    }.
\end{align}
To construct $\tilde{\mG}$, we first represent the double sum in \cref{eq:weak_volterra_basis_action} by a CFP series with coefficients $\tilde{A}_{r\ell}$. The following theorem gives closed-form expressions for $\tilde{A}_{r\ell}$.

\begin{theorem}\label{thm:weakly_singular_volterra_matrix}
The spectral approximation to the weakly singular Volterra integral operator \cref{eq:singular_volterra} can be formed as
\begin{align}\label{eq:weakly_singular_volterra_matrix}
        \tilde\mG
        =
        \tilde\mM\left[
            \left(\frac{1+x}{2}\right)^{1-\mu}
        \right]
        \tilde{\mA},
\end{align}
where $\tilde\mM[\bigl((1+x)/2\bigr)^{1-\mu}]$ denotes multiplication by $\bigl((1+x)/2\bigr)^{1-\mu}$ in the basis $\{Q_\ell^{\sigma,\gamma}\}_{\ell\geqslant0}$. The entries of $\tilde{\mA}=(\tilde{A}_{r,\ell})_{r,\ell\geqslant0}$ are
\begin{subnumcases}{
    \tilde{A}_{r\ell}
    =
    \label{eq:weak_A_entries}
    }
    \displaystyle
    \sum_{i=0}^{n}\delta_{ri}\sum_{j=0}^{m}
     \tilde{C}_{ij} \tilde{H}_{j,0,0},
    &
    \(\ell=0\),
    \label{eq:weak_A_entries_ell0}
    \\
    \displaystyle
    \frac{2-\delta_{r0}}{2}
    \sum_{i=0}^{n}\sum_{j=0}^{m}
    \tilde{C}_{ij}
    \left(
        \frac{\ell}{\ell+|i-r|}
        \tilde{H}_{j,|i-r|,\ell}
        +
        \frac{\ell}{\ell+i+r}
        \tilde{H}_{j,i+r,\ell}
    \right),
    &
    \(\ell\geqslant1\).
    \label{eq:weak_A_entries_ellpos}
\end{subnumcases}
For $0\leqslant k\leqslant\ell$, the quantity $\tilde{H}_{j,k,\ell}$ is defined by
\begin{align}\label{eq:weak_H_definition}
        \tilde{H}_{j,k,\ell}
        =
        \int_{-1}^{1}
        (1-t)^{-\mu}
        \left(\frac{1+t}{2}\right)^{k\gamma}
        Q_j^{\sigma,\gamma}(t)
        P_{\ell-k}^{(-1,2k)}
        \bigl(\varphi(t)\bigr)
        \,\md t,
\end{align}
whereas $\tilde{H}_{j,k,\ell}=0$ for $k>\ell$.
\end{theorem}

\begin{proof}
Equation \cref{eq:weakly_singular_volterra_matrix} follows immediately from \cref{eq:weak_volterra_basis_action}.

The coefficient $\tilde{A}_{r\ell}$ is obtained by taking the inner product of the double sum in \cref{eq:weak_volterra_basis_action} with $Q_r^{\sigma,\gamma}(x)$ with respect to $w_{\sigma, \gamma}(x)$. Using the CFP orthogonality \cref{eq:cfp_orthogonality_relation} and denoting $s=\varphi(x)$, we have
\begin{align}\label{eq:weak_A_coefficient_integral}
\tilde{A}_{r\ell}
    &=
    \frac{2-\delta_{r0}}{\pi}
    \sum_{i=0}^{n}\sum_{j=0}^{m} \tilde{C}_{ij}
    \int_{-1}^{1}
    (1-t)^{-\mu}Q_j^{\sigma,\gamma}(t)
    \left(
        \int_{-1}^{1}
        T_i(s)T_r(s)
        \bigl(
            \mC_{[-1,\varphi(t)]}T_\ell
        \bigr)(s)
        \frac{\md s}{\sqrt{1-s^2}}
    \right)
    \md t.
\end{align}
For $\ell=0$, $\mC_{[-1,\varphi(t)]}T_0(s)=1$, from which \cref{eq:weak_A_entries_ell0} follows.

For $\ell\geqslant1$, \cref{eq:chebyshev_product,eq:restriction_analytical_expression_chebyshev,eq:jacobi_minus_one_zero} lead to
\begin{align}
    &\int_{-1}^{1}
    T_i(s)T_r(s)
    \bigl(
        \mC_{[-1,\varphi(t)]}T_\ell
    \bigr)(s)
    \frac{\md s}{\sqrt{1-s^2}}
    \notag\\
    &=
    \frac{\pi}{2}
    \left(
        \frac{\ell}{\ell+|i-r|}
        \left(\frac{1+t}{2}\right)^{\gamma|i-r|}
        P_{\ell-|i-r|}^{(-1,2|i-r|)}
        \bigl(\varphi(t)\bigr)
        +
        \frac{\ell}{\ell+i+r}
        \left(\frac{1+t}{2}\right)^{\gamma(i+r)}
        P_{\ell-i-r}^{(-1,2(i+r))}
        \bigl(\varphi(t)\bigr)
    \right).
    \label{eq:weak_inner_chebyshev_coefficient}
\end{align}
When $\ell < |i-r|$ or $\ell < i+r$, the corresponding summand vanishes following a similar argument for \cref{thm:volterra_matrix}. Substituting \cref{eq:weak_inner_chebyshev_coefficient} into \cref{eq:weak_A_coefficient_integral} gives \cref{eq:weak_A_entries_ellpos}.
\end{proof}

\subsection{Sparse structure}
For the zeroth column of $\tilde{\mA}$, by \cref{eq:weak_A_entries_ell0} $\tilde A_{r0}=0$ for any $r>n$. For $\ell \geqslant 1$, \cref{eq:weak_A_entries_ellpos} implies that $\tilde A_{r\ell}=0$ whenever $i+r>\ell$ and $|i-r|=r-i>\ell$ hold simultaneously, i.e., $r>\ell+n$. Hence, the lower bandwidth of $\tilde{\mA}$ is at most $n$.

Since $1-\mu=p\gamma$,
\begin{align*}
    \left(\frac{1+x}{2}\right)^{1-\mu}
    =
    \left(\frac{1+\varphi(x)}{2}\right)^p.
\end{align*}
Thus, $\bigl((1+x)/2\bigr)^{1-\mu}$ is a polynomial of degree $p$ in $\varphi(x)$. Hence, $\tilde{\mM}[\bigl((1+x)/2\bigr)^{1-\mu}]$ has bandwidths at most $(p,p)$, implying that $\tilde{\mG}$ is lower banded with lower bandwidth at most $n+p$.

\subsection{Fast construction by recurrence}
As in the smooth linear case, the quantities $\tilde{H}_{j,k,\ell}$ can be computed using recurrence relations. We first note that the fiber
\begin{align}\label{eq:weak_H_initial_sequence}
\tilde{H}_{0,0,\ell}
    =
    2^{-\sigma}
    \int_{-1}^{1}
    (1-t)^{-\mu}
    (1+t)^\sigma
    P_\ell^{(-1,0)}\bigl(\varphi(t)\bigr)
    \,\md t
\end{align}
can be computed accurately using double-exponential quadrature~\cite{tak} and used as the initial values for recursing other fibers in the $k=0$ slice, as shown by the following theorem.

\begin{theorem}\label{thm:weakly_H_recurrence_k01}
\begin{subnumcases}{
    \tilde H_{1,0,\ell}
    =
    \label{eq:H10ell}
    }
    \displaystyle
    2\tilde H_{0,0,1} + \tilde H_{0,0,0}, & \(\ell=0\),
    \label{eq:tildeH100}
    \\
    \displaystyle
    \frac{\ell+1}{2\ell+1}\tilde H_{0,0,\ell+1}
    -
    \frac{1}{4\ell^2-1}
    \tilde H_{0,0,\ell}
    +
    \frac{\ell-1}{2\ell-1}
    \tilde H_{0,0,\ell-1},
    &
    \(\ell\geqslant1\).
    \label{eq:weak_H_recurrence_j1}
\end{subnumcases}
For $j\geqslant1$, 
\begin{subnumcases}{
    \tilde H_{j+1,0,\ell}
    =
    \label{eq:Hj+10ell}
    }
    \displaystyle
    4\tilde H_{j,0,1} + 2\tilde H_{j,0,0} - \tilde H_{j-1,0,0}, & \(\ell=0\),
    \\
    \displaystyle
    \frac{2(\ell+1)}{2\ell+1}
    \tilde H_{j,0,\ell+1}
    -
    \frac{2}{4\ell^2-1}
    \tilde H_{j,0,\ell}
    +
    \frac{2(\ell-1)}{2\ell-1}
    \tilde H_{j,0,\ell-1}
    -
    \tilde H_{j-1,0,\ell},
    &
    \(\ell\geqslant1\).
\end{subnumcases}
\end{theorem}

\begin{proof}
It follows from \cref{eq:jacobi_minus_one_zero,eq:jacobi_minus_one_left} that
\begin{align}\label{eq:P}
xP_0^{(-1, 0)}(x) = 2P_1^{(-1, 0)}(x)+P_0^{(-1, 0)}(x).
\end{align}
By definition,
\begin{align*}
\tilde{H}_{1,0,0} = \int_{-1}^{1} (1-t)^{-\mu} \left(\frac{1+t}{2}\right)^{\sigma} \varphi(t) P_0^{(-1,0)} \bigl(\varphi(t)\bigr) \,\md t,
\end{align*}
which, upon substitution of \cref{eq:P}, gives \cref{eq:tildeH100}.

Combining \cref{eq:jacobi_minus_one_left} and \cref{eq:jacobi_multiplication} yields
\begin{align}\label{eq:intermediate}
\varphi(t) P^{(-1,0)}_{\ell}\bigl(\varphi(t)\bigr)
    &=
    \frac{\ell+1}{2\ell+1}P^{(-1,0)}_{\ell+1}\bigl(\varphi(t)\bigr)
    -
    \frac{1}{4\ell^2-1} P^{(-1,0)}_{\ell}\bigl(\varphi(t)\bigr)
    +
    \frac{\ell-1}{2\ell-1} P^{(-1,0)}_{\ell-1}\bigl(\varphi(t)\bigr),
\end{align}
where $\ell \geqslant 1$. Multiplying \cref{eq:intermediate} by $(1-t)^{-\mu} \bigl((1+t)/2\bigr)^{\sigma}$ and integrating over $[-1, 1]$ give \cref{eq:weak_H_recurrence_j1}.

Equation \cref{eq:Hj+10ell} follows analogously.
\end{proof}

Once the zeroth slice is available, the rest of the $\tilde{H}$ tensor could be obtained recursively for increasing $k \geqslant 1$ using the following theorem.

\begin{theorem}\label{thm:weakly_H_recurrence}
For $j\geqslant0$, $k\geqslant0$, and $\ell\geqslant k+1$,
\begin{align}\label{eq:weak_H_kell_recurrence}
    \tilde H_{j,k+1,\ell}
    =
    \frac{
        (\ell-k)\tilde H_{j,k,\ell}
        +
        (\ell+k)\tilde H_{j,k,\ell-1}
        -
        (\ell-k-2)\tilde H_{j,k+1,\ell-1}
    }{
        \ell+k
    }.
\end{align}
\end{theorem}
\begin{proof}
We first consider the case of $\ell=k+1$. It follows from
\cref{eq:phi,eq:jacobi_minus_one_zero,eq:jacobi_minus_one_left} that
\begin{align}\label{eq:P0P1}
\left( \frac{1+t}{2}\right)^{\gamma} P_0^{(-1, 2k+2)}\bigl(\varphi(t)\bigr) = \frac{1}{2k+1}P_1^{(-1,2k)}\bigl(\varphi(t)\bigr)+P_0^{(-1,2k)}\bigl(\varphi(t)\bigr).
\end{align}
Multiplying \cref{eq:P0P1} by $(1-t)^{-\mu}\bigl((1+t)/2\bigr)^{k\gamma}Q_j^{\sigma,\gamma}(t)$ and integrating over $[-1, 1]$ give 
\begin{align*}
    \tilde H_{j,k+1,\ell}
    =
    \frac{
        (\ell-k)\tilde H_{j,k,\ell}
        +
        (\ell+k)\tilde H_{j,k,\ell-1}
        }{\ell+k}.
\end{align*}
We immediately have \cref{eq:weak_H_kell_recurrence} by noting $\tilde H_{j,k+1,k}=0$.

For $\ell\geqslant k+2$, we specialize \cref{eq:jacobi_parameter_shift_left} to $\beta=2k+1$ and $n=\ell-k-1,\ell-k-2$ and \cref{eq:jacobi_parameter_shift_right} to $\beta=2k$ and $n=\ell-k,\ell-k-1$. Combining the four resulting identities gives
\begin{align*}
(\ell+k)\frac{1+s}{2}
    P_{\ell-k-1}^{(-1,2k+2)}(s)
    +
    (\ell-k-2)\frac{1+s}{2}
    P_{\ell-k-2}^{(-1,2k+2)}(s)
    =
    (\ell-k)P_{\ell-k}^{(-1,2k)}(s)
    +
    (\ell+k)P_{\ell-k-1}^{(-1,2k)}(s),
\end{align*}
where $s=\varphi(t)$. Using the identity $\bigl(1+\varphi(t)\bigr)/2=\bigl((1+t)/2\bigr)^\gamma$, multiplying both sides by $(1-t)^{-\mu}\bigl((1+t)/2\bigr)^{\sigma+k\gamma}Q_j^{0,\gamma}(t)$, and integrating over $[-1,1]$ lead to \cref{eq:weak_H_kell_recurrence}.
\end{proof}

Introducing the contracted quantities
\begin{align}\label{eq:weak_contractions}
\tilde{\Lambda}_{i,k,\ell} = \sum_{j=0}^{m} \tilde{C}_{ij}\tilde H_{j,k,\ell},
\end{align}
we can rewrite \cref{eq:weak_A_entries} as
\begin{subnumcases}{
    \tilde{A}_{r\ell}
    =
    \label{eq:weak_A_entries_Lambda}
    }
    \displaystyle \sum_{i=0}^n \delta_{ri} \tilde{\Lambda}_{i,0,0}, & \(\ell=0\),
    \\
    \displaystyle
    \frac{2-\delta_{r0}}{2}
    \sum_{i=0}^{n}
    \left(
        \frac{\ell}{\ell+|i-r|}
        \tilde{\Lambda}_{i,|i-r|,\ell}
        +
        \frac{\ell}{\ell+i+r}
        \tilde{\Lambda}_{i,i+r,\ell}
    \right),
    &
    \(\ell\geqslant1\).
\end{subnumcases}
A key observation is that $\tilde{\Lambda}_{i,k,\ell}$ satisfies the same recurrence relation as \cref{eq:weak_H_kell_recurrence}.

\begin{corollary}\label[corollary]{cor:weak_Lambda_recurrence}
For $0\leqslant i\leqslant n$, $k\geqslant0$, and $\ell\geqslant k+1$,
\begin{align}\label{eq:weak_Lambda_kell_recurrence}
    \tilde{\Lambda}_{i,k+1,\ell}
    =
    \frac{
        (\ell-k)\tilde{\Lambda}_{i,k,\ell}
        +
        (\ell+k)\tilde{\Lambda}_{i,k,\ell-1}
        -
        (\ell-k-2)\tilde{\Lambda}_{i,k+1,\ell-1}
    }{
        \ell+k
    }.
\end{align}
\end{corollary}
\begin{proof}
The result follows by multiplying \cref{eq:weak_H_kell_recurrence} by $\tilde{C}_{ij}$ and summing over $j$.
\end{proof}

This corollary suggests that instead of computing $\tilde{H}_{j,k,\ell}$ by \cref{eq:weak_H_kell_recurrence} and forming $\tilde{\mA}$ following \cref{eq:weak_A_entries} we can compute $\tilde{H}_{j,0,\ell}$ by \cref{eq:H10ell,eq:Hj+10ell} and $\tilde{\Lambda}_{i,0,\ell}$ by the definition, recurse with \cref{eq:weak_Lambda_kell_recurrence} for the remaining slices of $\tilde{\Lambda}_{i,k,\ell}$, and obtain $\tilde \mA$ by \cref{eq:weak_A_entries_Lambda}. This dispenses with the calculation of $\tilde H_{j,k,\ell}$ for $k \geqslant 1$ and the summation over the index $j$ in \cref{eq:weak_A_entries_ellpos}. The strategy is illustrated in \Cref{fig:weak-fast-recurrence}, and the complete algorithm for computing the $N \times N$ section of $\tilde{\mG}$ is given by \Cref{alg:weakly_singular_volterra_matrix}. The overall complexity is $\mO(nN^2)$, provided that $N \gg m,n,p$. If $\tilde{\mA}$ were constructed via \cref{eq:weak_H_kell_recurrence}, the cost would be $\mO(mnN^2)$.

\begin{figure}[t]
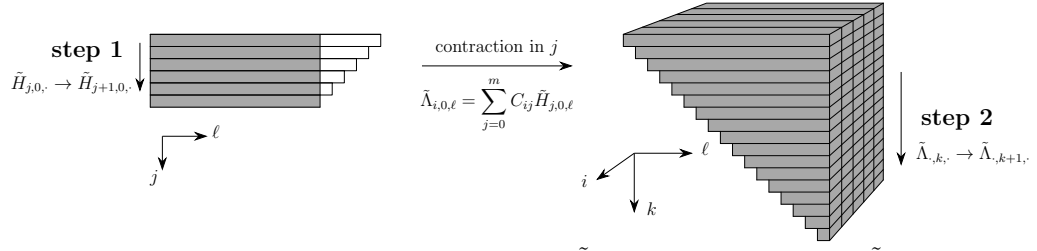

\centering
\WeakFastRecurrenceGraphic
\caption{Construction of the contracted moments $\tilde\Lambda_{i,k,\ell}$. We first compute $\tilde H_{j,0,\ell}$ of the $k=0$ slice by recursing with $j$. The auxiliary moments are unshaded. After forming $\{\tilde{\Lambda}_{i,0,\ell}\}_{i,\ell}$, we recurse for $\tilde\Lambda_{i,k,\ell}$ for incremental $k \geqslant 1$.}
\label{fig:weak-fast-recurrence}
\end{figure}

\begin{algorithm}[t]
\caption{Construction of the spectral approximation to a weakly singular linear Volterra operator}
\label{alg:weakly_singular_volterra_matrix}
\begin{algorithmic}[1]

\REQUIRE
Kernel $K(x,y)$, parameters $\mu$, $\sigma$, $\gamma$, $p=\frac{1-\mu}{\gamma}\in\mathbb N_+$, quadrature size $q$, truncation size $N$.

\ENSURE
The $N\times N$ principal finite section of the spectral approximation $\tilde{\mG}$.

\STATE Compute the CFP coefficients $\{\tilde{C}_{ij}\}_{i,j=0}^{n,m}$ in \cref{eq:weak_kernel_cfp_expansion}.
\hfill $\triangleright\ \mO\!\bigl(nm(\log n+\log m)\bigr)$

\STATE Compute $\tilde H_{0,0,\ell}$ in \cref{eq:weak_H_initial_sequence} by double-exponential quadrature.
\hfill $\triangleright\ \mO\!\bigl(q(N+m)\bigr)$

\STATE Compute $\tilde H_{j,0,\ell}$ using \cref{eq:H10ell,eq:Hj+10ell}.
\hfill $\triangleright\ \mO\!\bigl(m(N+m)\bigr)$

\STATE Form $\tilde\Lambda_{i,0,\ell}$ following \cref{eq:weak_contractions}.
\hfill $\triangleright\ \mO(nmN)$

\STATE Generate $\tilde\Lambda_{i,k,\ell}$ by recursion using \cref{eq:weak_Lambda_kell_recurrence}.
\hfill $\triangleright\ \mO(nN^2)$ 

\STATE Assemble $\tilde{\mA}$ using \cref{eq:weak_A_entries_Lambda}.
\hfill $\triangleright\ \mO(nN^2)$

\STATE Form $\tilde{\mG}$ by \cref{eq:weakly_singular_volterra_matrix}.
\hfill $\triangleright\ \mO(pN^2)$
\end{algorithmic}
\end{algorithm}

\section{Numerical examples}\label{sec:numerical_experiments}
All numerical experiments in this section are performed in Julia 1.11.1 on a computer with a 13th Gen Intel(R) Core(TM) i7-13650HX processor at 2.60 GHz and 16 GB of memory. The reported timings are measured using the package \texttt{BenchmarkTools.jl}. The almost-banded systems are solved using the package \texttt{FastAlmostBandedMatrices.jl}~\cite{fas}, and throughout we take $m = n$ for simplicity of discussion.

The \texttt{VoltOpt} repository~\cite{che2} contains the implementations of the proposed and benchmark methods and the code for the experiments in this section. In the rest of this section, Gutleb--Olver denotes the implementation of the algorithm in~\cite{gut2} based on the publicly available code~\cite{gut3}, with some further optimization. M{\"u}ntz--collocation denotes our implementation of the method in~\cite{hou}. The proposed spectral approximations to Volterra integral operators with smooth and weakly singular kernels are denoted by Duffy--Chebyshev and Duffy--CFP, respectively.

We denote by $u_*(x)$ the exact solution of the problem under consideration.

\subsection{A linear Volterra equation with a high-degree solution}\label{sec:example_2}
Our first example is the second-kind Volterra integral equation
\begin{equation*}
u(x)+\int_{-1}^{x}K(x,y)u(y)\,\md y=f(x),
\qquad x \in[-1, 1],
\end{equation*}
where
\begin{equation*}
K(x,y)
=
\frac{1}{\tau(x,y)}
\exp\left(-\frac{x-y}{\tau(x,y)}\right),
\qquad
\tau(x,y)=\frac32+x-\frac{y}{2}.
\end{equation*}
Although $K(x,y)$ is smooth on $T^2$, where $\tau(x,y)\geqslant1$, it is singular at $(-1,1)$. This problem cannot therefore be solved by methods that rely on extending the closed-form expression to the non-causal triangle \cite{hal1,dri}.

The right-hand side is chosen so that the exact solution is
\begin{equation*}
u_\ast(x)
=
1+a_1 S_1(x-x_1)+ a_2 S_2(x-x_2),
\qquad
S_k(z)=\frac12\left(1+\tanh\frac{z}{\delta_k}\right)
\end{equation*}
with $x_1=-0.35$, $x_2=0.42$, $a_1=1$, $a_2=-0.6$, $\delta_1=10^{-3}$, and $\delta_2=1.5\times10^{-3}$. The small values of $\delta_1$ and $\delta_2$ produce two narrow transition layers, which test the efficiency of the proposed method as a high truncation degree is required for adequate resolution.

With the approximation to the smooth linear Volterra operator constructed by \Cref{alg:cheb_volterra_matrix}, the almost-banded system to solve is
\begin{equation*}
(\mI+\mG)\bu=\bf,
\end{equation*}
where $\bf$ denotes the coefficient vector of $f(x)$.

The computed solution is evaluated by the Clenshaw algorithm and compared with $u_*(x)$ on a uniform grid of $10^4$ equispaced points on $[-1,1]$. As shown in \Cref{fig:ex2_error}, the maximum error obtained on this grid decays exponentially for both methods. However, for the same $N$, the solution obtained by the proposed method is a few orders of magnitude more accurate than that obtained by the Gutleb--Olver method. Whereas the convergence of the Gutleb--Olver method ceases to improve at about $\mO(10^{-13})$, the error of the proposed method levels off at approximately machine precision. To fully resolve the solution at machine precision, the computed solution must be a Chebyshev series of degree higher than about $18,900$, due to the transition layers.

As shown by \Cref{fig:ex2_time}, the proposed method is roughly $10\times$ faster throughout the selected range of $N \in [10^2, 2\times 10^4]$. The reported times include both the construction and solution of the linear system.

\begin{figure}[htpb]
\centering
\subfloat[]{
\includegraphics[width=0.45\linewidth]{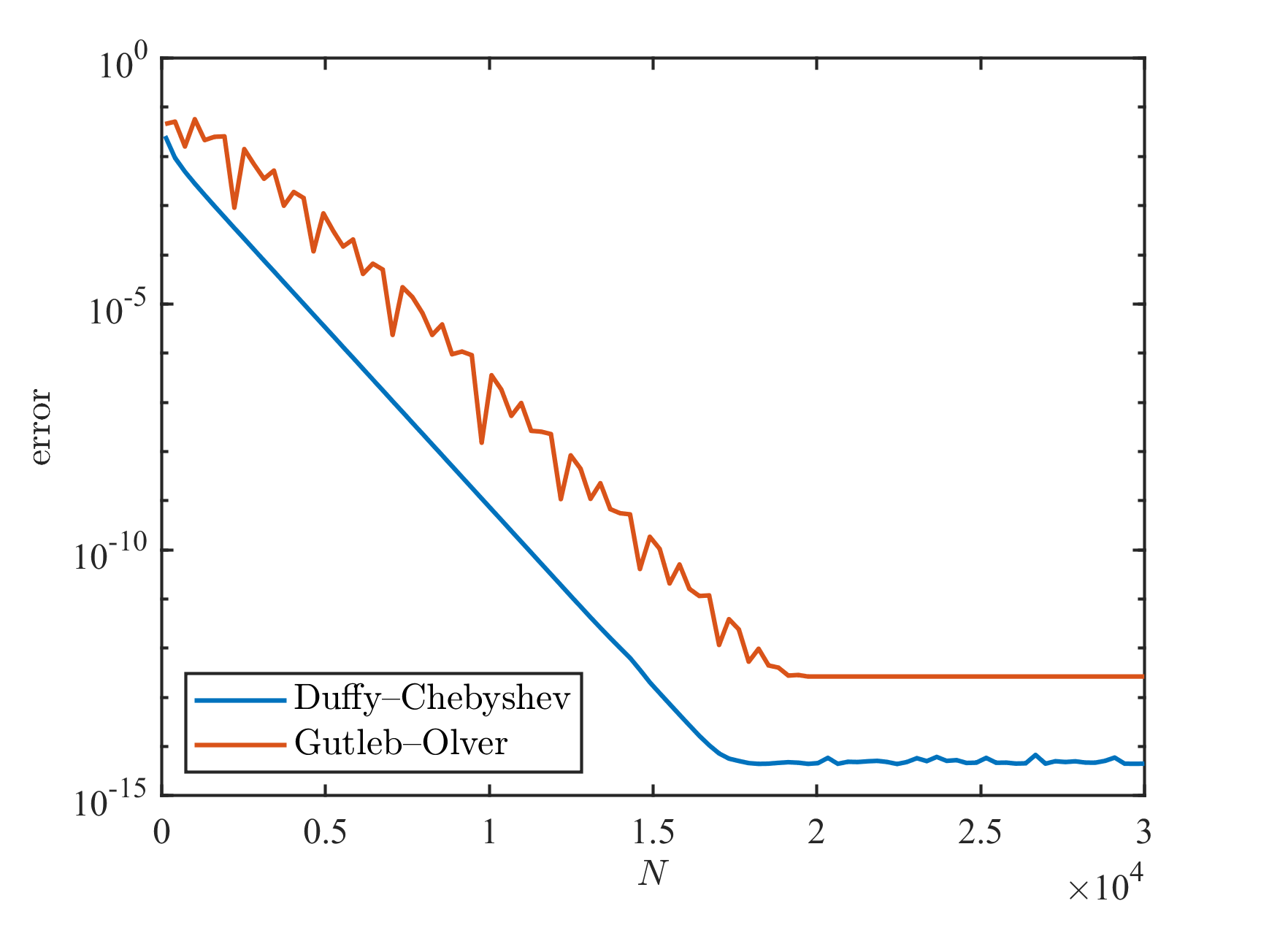}
\label{fig:ex2_error}
}
\hspace{0.1cm}
\subfloat[]{
\includegraphics[width=0.45\linewidth]{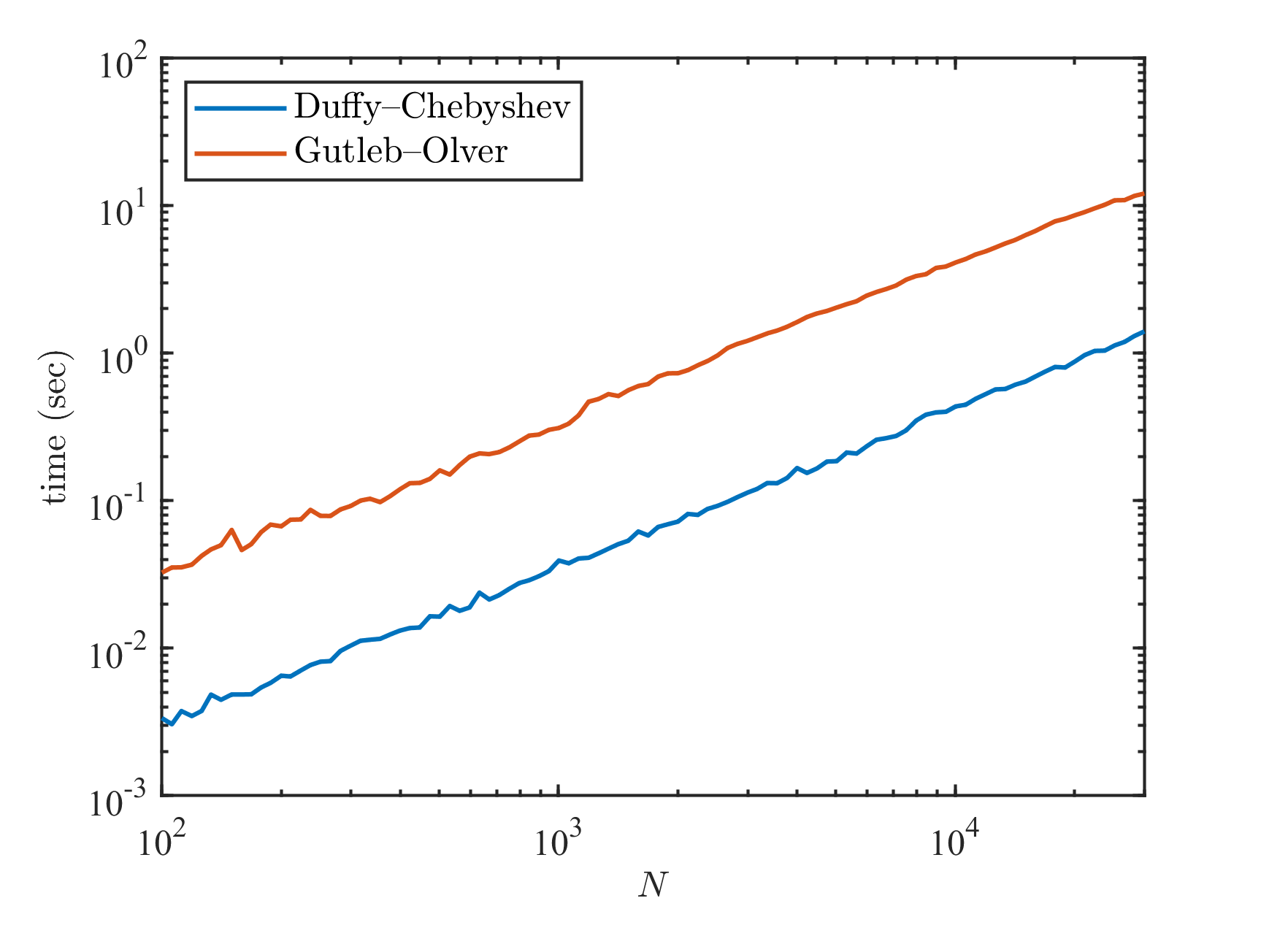}
\label{fig:ex2_time}
}
\caption{
(A) maximum pointwise errors versus $N$. (B) total execution times.
}
\label{fig:example_2}
\end{figure}

\subsection{A second-kind Volterra equation with an oscillatory kernel}\label{sec:example_3}
Consider the second-kind Volterra integral equation
\begin{align*}
u(x)+\int_0^x \cos(\omega xy)u(y)\,\md y=f(x),
\qquad x\in[0,1],
\end{align*}
where $\omega=20$ and
\begin{align*}
f(x)
=
\cos(50x)
+\frac{x}{2}
\Bigl(
\operatorname{sinc}\bigl((\omega x-50)x\bigr)
+
\operatorname{sinc}\bigl((\omega x+50)x\bigr)
\Bigr),
\end{align*}
where $\operatorname{sinc}(x) = \sin(x) / x$. The exact solution is $u_*(x)=\cos(50x)$.

We fix the truncation size at $N=200$, which is sufficient for the solution to be fully resolved, and vary the truncation size $n$ of the bivariate approximant from $2$ to $60$. For each $n$, the computed solutions are evaluated on $2\times 10^4$ equispaced points in $[0,1]$. As shown in \Cref{fig:ex3_error}, the proposed method is also spectrally convergent with respect to the truncation size of the bivariate approximant $n$ until the error plateaus at machine precision. This is entirely as expected, since the transformed kernel functions are analytic and hence admit exponentially convergent bivariate Chebyshev approximations.

Since the truncation size $n$ of the bivariate Chebyshev approximant to the transformed kernel function coincides with that of the series of bivariate polynomials orthogonal on $T^2$ that approximates the original kernel function on $T^2$, we perform the same experiment for the Gutleb--Olver method and show the result in the same plot. Unlike the proposed method, the Gutleb--Olver solution converges only up to $\mO(10^{-7})$ and quickly diverges afterwards. This surprising instability is in fact caused by the operator-valued Clenshaw algorithm used in the Gutleb--Olver construction, where the approximation errors in the kernel coefficients accumulate and are amplified as Clenshaw's algorithm progresses.

\Cref{fig:ex3_time} shows that for fixed $N$ the total execution times of the Gutleb--Olver and proposed methods scale as $\mO(n^3)$ and $\mO(n^2\log n)$, respectively, with the construction costs dominating the total cost of both methods. The kinks in the blue curve are attributed to the implementation of FFT in the FFTW library.

\begin{figure}[htpb]
\centering
\subfloat[]{
\includegraphics[width=0.45\linewidth]{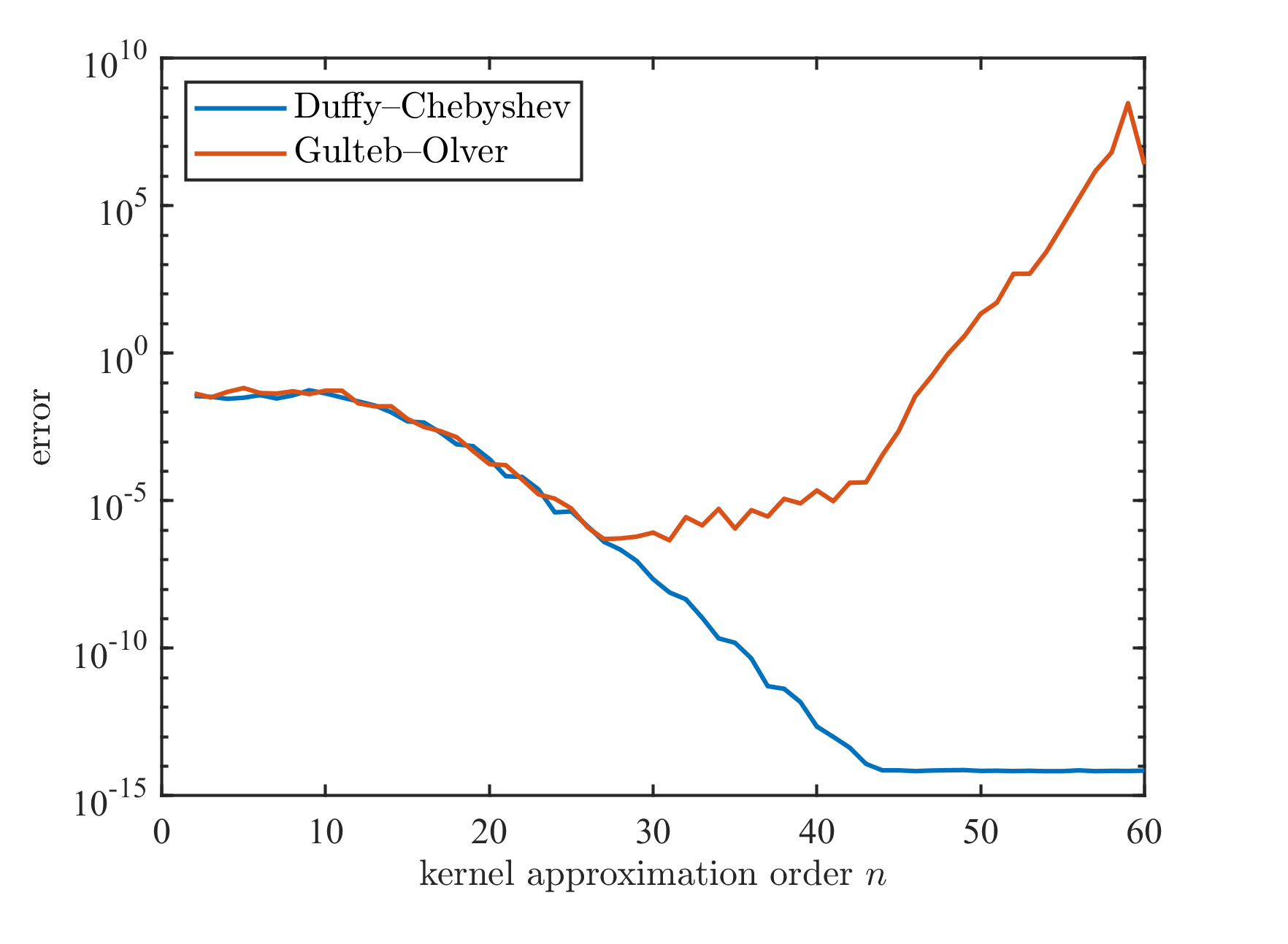}
\label{fig:ex3_error}
}
\hspace{0.1cm}
\subfloat[]{
\includegraphics[width=0.45\linewidth]{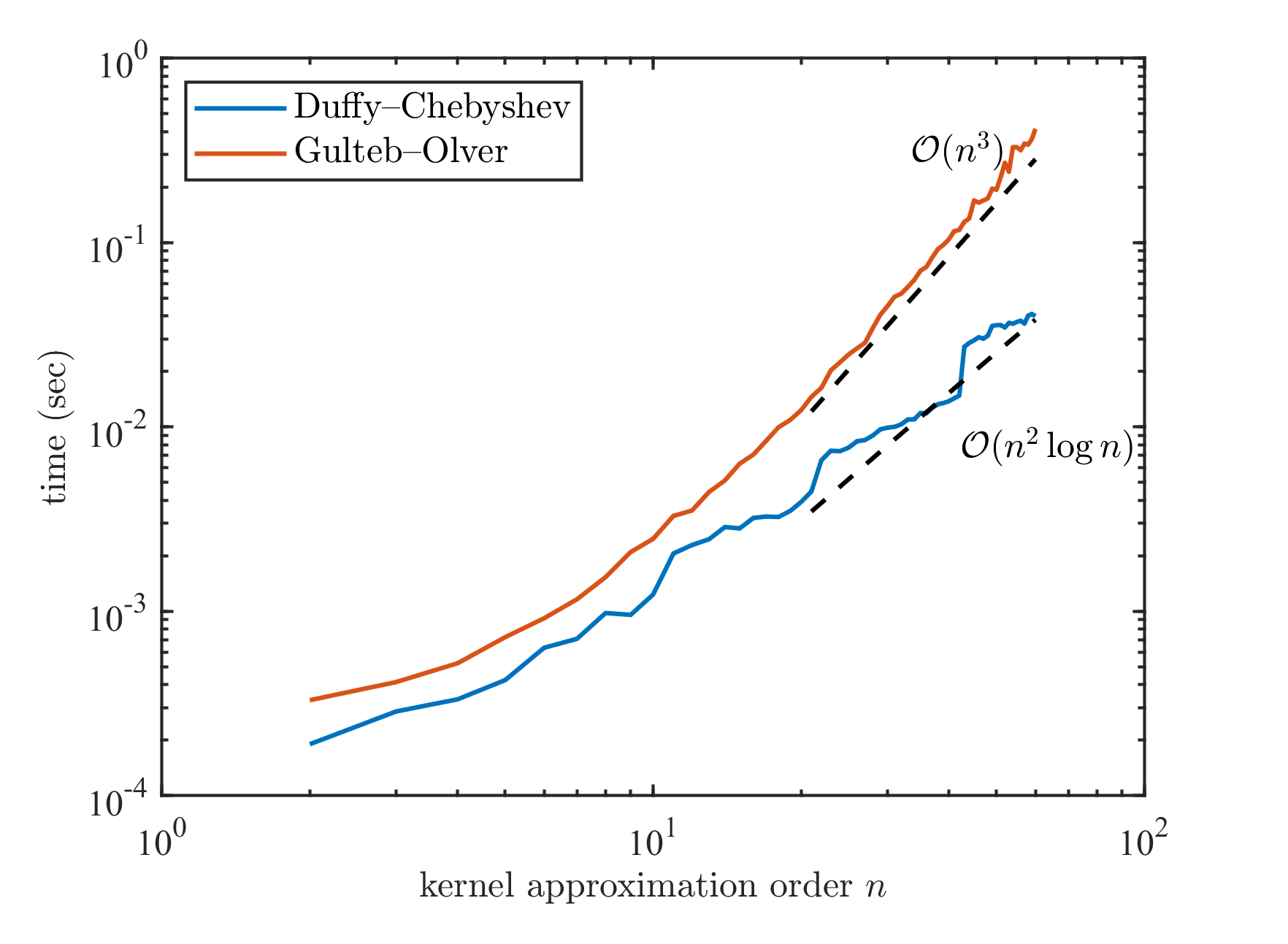}
\label{fig:ex3_time}
}
\caption{
(A) maximum errors versus the kernel approximation degree. (B) total execution times.
}
\label{fig:example_3}
\end{figure}

\subsection{A first-kind pantograph Volterra integral equation}\label{sec:example_4}
The third example is a pantograph Volterra convolution integral equation of the first kind
\begin{equation}\label{eq:proportional_delay_first_kind}
\int_{qx}^{x}K(x,y)u(y)\,\md y=f(x), \qquad x\in[0,1],
\end{equation}
where $0< q <1$. Such equations go back to Volterra~\cite{vol} and were extensively studied by Brunner~\cite{bru3,bru2}; more recent numerical treatments include~\cite{hal2,xie}. In this example, we take $K(x,y)=\exp\bigl(-(x-y)^2\bigr)$ and $q=0.5$, and the right-hand side is chosen so that the exact solution is $u_*(x)=e^x\cos(20x)$. As a Volterra convolution equation, \cref{eq:proportional_delay_first_kind} can be solved by replacing the Volterra convolution operator with its infinite-dimensional spectral approximation and $f(x)$ with the corresponding Chebyshev coefficient vector and solving the resulting linear system adaptively \cite{xu}. However, a numerical solution of a first-kind integral equation is rarely stable. Therefore, we differentiate \cref{eq:proportional_delay_first_kind}, split the resulting integral at the origin, and apply the change of variables $y=qs$. This gives rise to the Volterra integral equation of the second kind
\begin{align*}
K(x,x)u(x)
-qK(x,qx)u(qx)
+
\int_0^x \partial_xK(x,s)u(s)\,\md s
-
q\int_0^x \partial_xK(x,qs)u(qs)\,\md s
=
f'(x).
\end{align*}
The first integral remains convolutional, whereas the scaled kernel in the second integral depends on $x-qs$ and is not convolutional. The proportional delay also appears through $u(qx)$ and $u(qs)$.

We approximate $K(x, x)$ and $K(x, qx)$ by univariate Chebyshev series and construct the corresponding multiplication matrices $\mM^{\dagger}$ and $\mM^{\ddagger}$ respectively. Next, we construct the spectral approximations to the Volterra integral operators defined by $\partial_xK(x,s)$ and $\partial_xK(x,qs)$ following \Cref{alg:cheb_volterra_matrix}, respectively, as well as the spectral approximation to the restriction operator that restricts a function defined on $[0, 1]$ to $[0, q]$ following \cref{eq:five_point_recurrence}. Denoting these approximations by $\mV^{\dagger}$, $\mV^{\ddagger}$, and $\mR$ respectively, we have the infinite-dimensional linear system
\begin{align*}
\left(\mM^{\dagger} - q\mM^{\ddagger}\mR + \mV^{\dagger} -q \mV^{\ddagger}\mR\right)\mathbf{u} = \mathbf{f},
\end{align*}
where $\mathbf{u}$ and $\mathbf{f}$ are the Chebyshev coefficient vectors of the solution and $f'(x)$ respectively. Since this linear system is lower-banded, it can be solved by adaptive QR for an optimal truncation of the solution vector. Whereas the state-of-the-art method incurs a cost of $\mO(N^3)$ for solving such problems \cite{hal2}, both the construction and solution complexities are $\mO(N^2)$ using the proposed approximation.

\Cref{fig:ex4_error} shows the maximum error obtained on a grid of $2\times 10^3$ equispaced points in $[0,1]$. The computed solution converges super-geometrically until the error levels off at machine precision.

\begin{figure}[htpb]
\centering
\subfloat[]{
\includegraphics[width=0.45\linewidth]{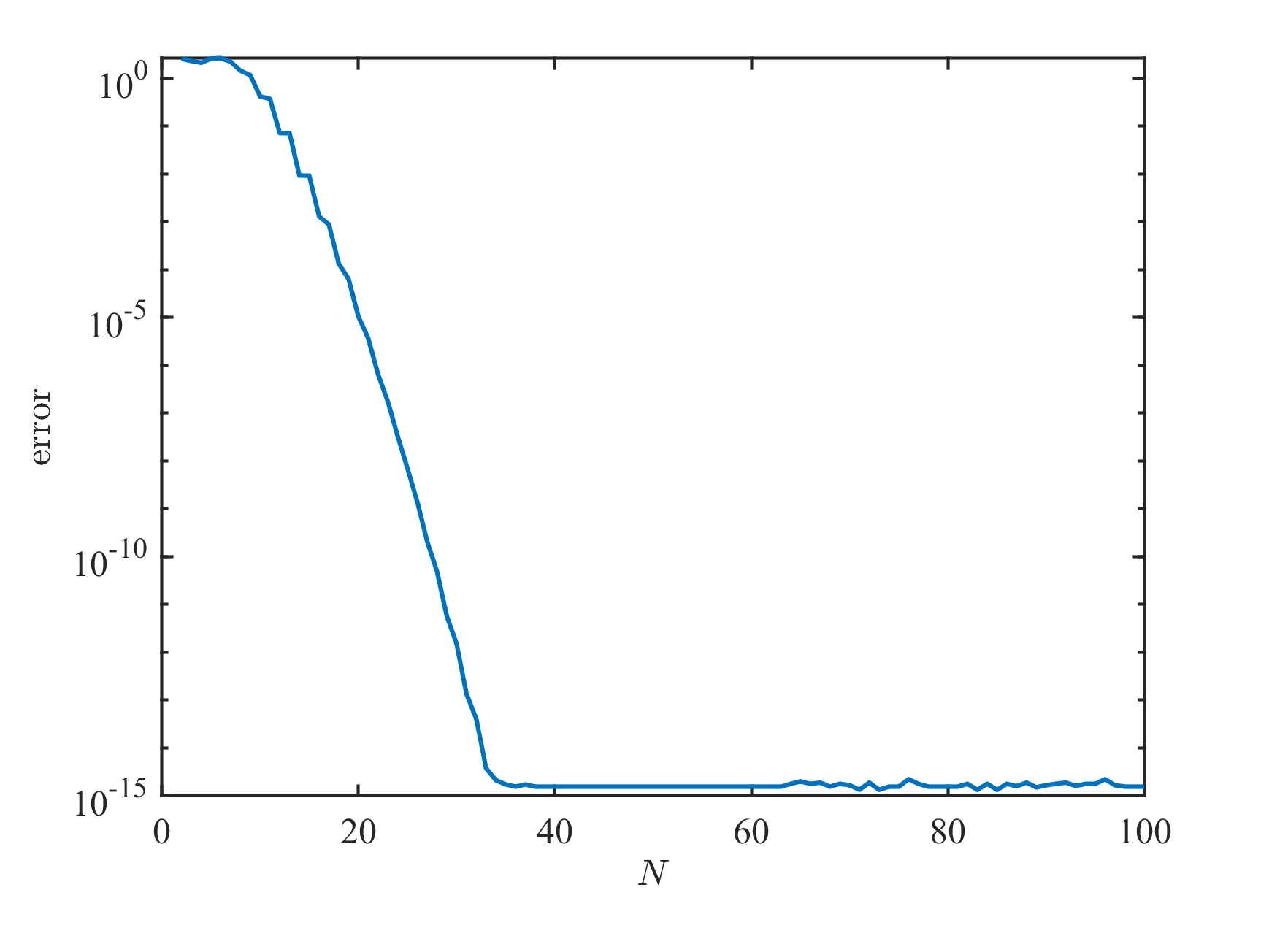}
\label{fig:ex4_error}
}
\hspace{0.1cm}
\subfloat[]{
\includegraphics[width=0.45\linewidth]{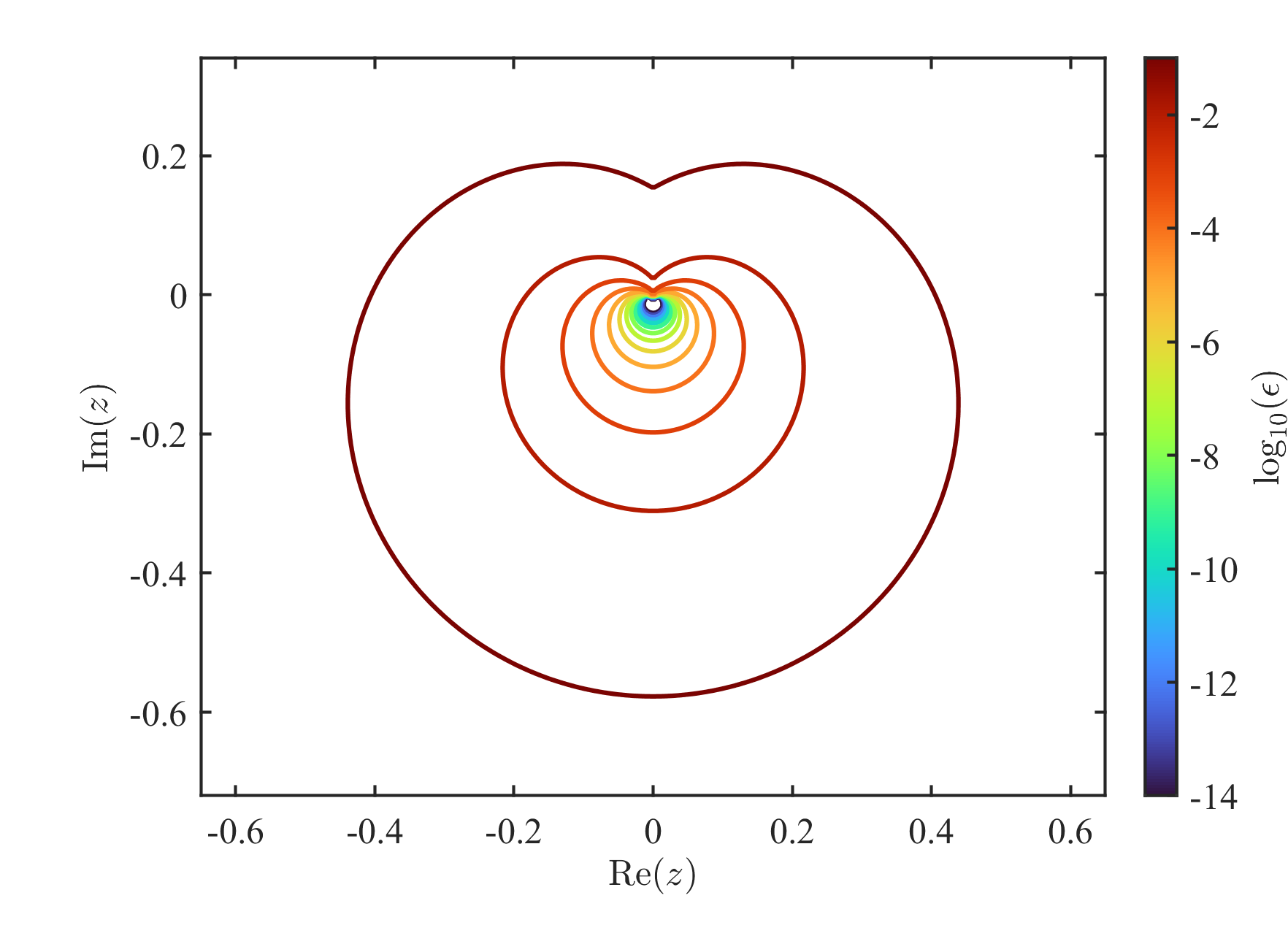}
\label{fig:ex5_pseudospectrum}
}
\caption{
(A) the maximum absolute error versus the truncation degree for the pantograph Volterra integral equation. (B) $\varepsilon$-pseudospectra of \cref{eq:smooth_volterra} with $K(x,y)=-\mathrm{i}J_1\bigl(2(x-y)+xy\bigr)$ for $\varepsilon = 10^{-1}, 10^{-2}, \ldots, 10^{-14}$.} \label{fig:example_45}
\end{figure}

\subsection{Pseudospectra of a Volterra integral operator}\label{sec:example_5}
A Volterra integral operator with a smooth kernel is quasinilpotent, and $0$ is the only element in its spectrum \cite{bru3}. The spectrum alone, however, offers little insight into the operator's behavior due to its nonnormality. In such scenarios, one has to resort to its pseudospectra \cite{tre}.

Consider the linear Volterra integral operator $\mV$ in \cref{eq:smooth_volterra} with
\begin{align*}
K(x,y) = -\mathrm{i}J_1\bigl(2(x-y)+xy\bigr).
\end{align*}
Here, $J_1(x)$ is the Bessel function of the first kind and $\mathrm{i}$ is the imaginary unit. We first construct the spectral approximation to $\mV$ in the Legendre basis, since the pseudospectra are usually considered with respect to the standard $L^2$ norm and the Legendre polynomials are orthogonal in $L^2([-1,1])$. The construction is analogous to the Chebyshev case, with implementation details available in the companion repository \cite{che2}. We then numerically approximate the resolvent norm of $\mV$ on a $320\times300$ equispaced tensor grid in the rectangular domain $[-0.65, 0.65]\times [-0.72\mathrm{i}, 0.34\mathrm{i}]$. The computed resolvent norms are fed into a contour plotter for the level curves shown in \Cref{fig:ex5_pseudospectrum}. These are the $\varepsilon$-pseudospectra of $\mV$ for $\varepsilon = 10^{-1}, 10^{-2}, \ldots, 10^{-14}$. The deformation of the contours from concentric circles illustrates the nonnormality of $\mV$, explaining the limited usefulness of eigenvalues or the spectrum in understanding this nonnormal Volterra integral operator.

\subsection{A nonlinear Volterra integral equation}\label{sec:example_6}
Consider the nonlinear Volterra integral equation of Urysohn type
\begin{align}
u(x)
=
g(x)
+
\int_{-1}^{x}
\frac{u(t)}
{1+(x-t)u^2(t)}
\,\md t,
\qquad x\in[-1,1],
\label{eq:nonlinear_volterra_example}
\end{align}
where $g(x)$ is chosen so that the exact solution is $u_*(x)=\sqrt{W(e^{\alpha x})}$. Here, $W(\cdot)$ is the Lambert $W$ function. Notably, the nonlinear kernel has a saturating dependence on both the solution and the time lag. Nonlinear Volterra--Urysohn integral equations arise in models with nonlinear history dependence; related theory and numerical methods can be found in~\cite{bru3}. Since the solution is coupled with the kernel function, such a problem is intractable for the method proposed in \cite{gut2}.

We take $\alpha=100$, which results in a rapidly varying solution and therefore requires a high-degree approximation. The problem is solved by adaptive Picard iteration, initialized with the Chebyshev series approximating $g(x)$. The solution progress is shown in \Cref{fig:nonlinear_picard_iteration}, where the horizontal axis denotes the number of iterations, including both the inner iterations for adaptively determining $N$ and the outer Picard iterations for resolving the nonlinearity \cite{qin}. The blue curve, with the left vertical axis as reference, shows the convergence of the Picard iteration. As indicated by the dots, $14$ Picard iterations are performed, whereas in the remaining inner iterations only the truncation degree $N$ of the solution is doubled to improve its resolution. The red curve, with the right vertical axis as reference, shows the adaptive growth of $N$---from $2^5$ to $2^{10}$ over the first five inner iterations. The final maximum error, obtained on an equispaced grid of $10^3$ points, is $7.11\times10^{-15}$.

\Cref{fig:nonlinear_time} compares the execution times of the proposed method for fast evaluation of nonlinear Volterra integral and the conventional method based on Clenshaw's algorithm for the nonlinear integral in \cref{eq:nonlinear_volterra_example}. The timings are consistent with the $\mO(N^2\log N)$ and $\mO(N^3)$ complexities. The speedup is $90\times$ at $N=2^{10}$.

\begin{figure}[htpb]
\centering
\subfloat[]{\includegraphics[width=0.45\linewidth]{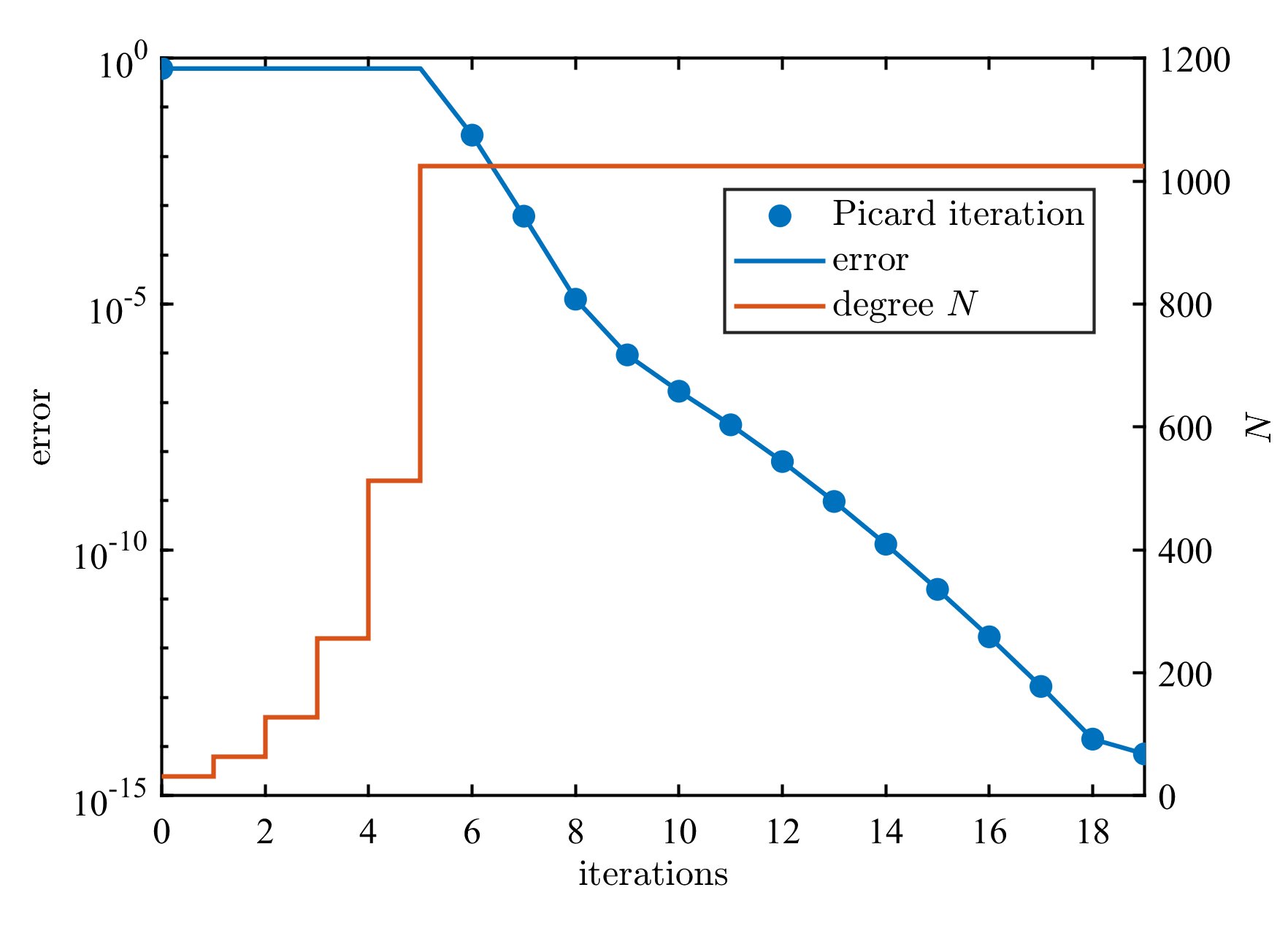}\label{fig:nonlinear_picard_iteration}}
\hspace{0.1cm}
\subfloat[]{\includegraphics[width=0.45\linewidth]{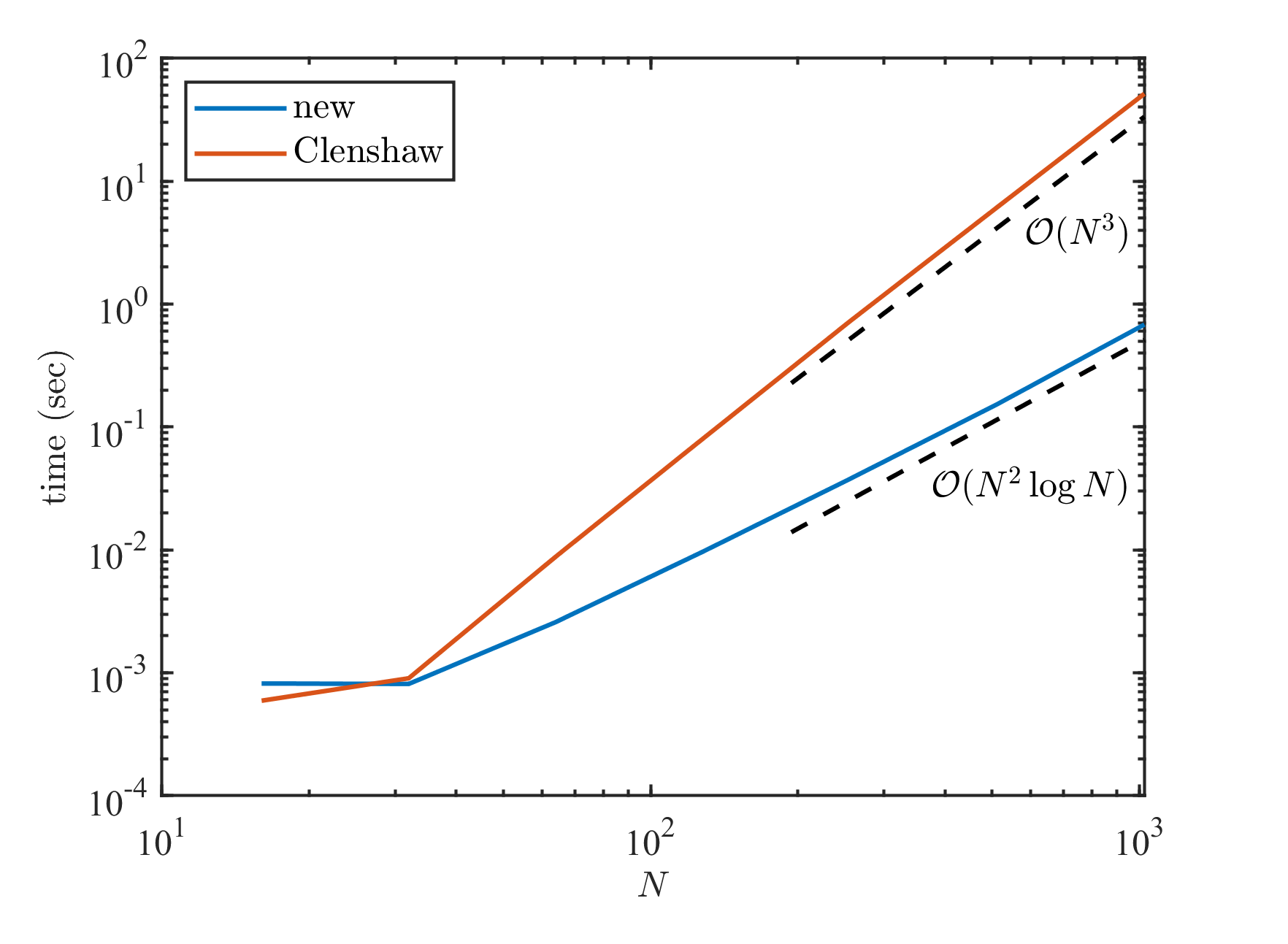}\label{fig:nonlinear_time}}
\caption{(A) History of the Picard iteration. (B) Execution time of the proposed method for evaluating the nonlinear Volterra integral in \cref{eq:nonlinear_volterra_example}, compared with the conventional Clenshaw-based method.}
\label{fig:example_6}
\end{figure}

\subsection{A weakly singular Volterra integral equation}\label{sec:example_7}
Our last example is the weakly singular Volterra equation
\begin{align*}
u(x)
+
\frac{1}{2}\int_{-1}^{x}
(x-y)^{-1/2}K(x,y)u(y)\,\md y
=
g(x),
\qquad x\in[-1,1],
\end{align*}
where
\begin{align*}
K(x,y)=\frac{1}{(x+1)^2+(y-1)^2},
\qquad
g(x)=
\left(\frac{1+x}{2}\right)^{1/3}
\left(
1+\frac{\sin(\pi x)}{10^3x^2+1}
\right).
\end{align*}
Although $K(x,y)$ is smooth on the causal domain $T^2$, it has a singularity at $(-1,1)$, which prevents the problem from being tackled by methods relying on extension.

We use the CFP basis with $\sigma=1/3$ and $\gamma=1/2$ and benchmark against the M{\"u}ntz--collocation method \cite{hou}, for which we set the parameter $\lambda=1/6$. As a closed-form solution is unavailable, the convergence is assessed by a value-based Cauchy error. For each computed solution of degree $N$, we also compute a solution of degree $N^+=\max(N+1,\lfloor 1.01N\rceil)$. The two approximate solutions are evaluated on a common test grid of $10^4$ equispaced points in $x\in[-1,1]$, and $\max_j|u_N(x_j)-u_{N^+}(x_j)|$ is taken as the Cauchy error. As shown in \Cref{fig:ex7_error}, though both methods exhibit geometric convergence, the proposed method converges faster and is a few orders of magnitude more accurate than the M{\"u}ntz--collocation method at the same $N$.

\Cref{fig:ex7_time} shows the execution times versus $N$, consistent with the predicted $\mO(N^2)$ scaling. The breakeven point of the two methods occurs at about $N = 100$, which is favorably small.

\begin{figure}[htpb]
\centering
\subfloat[]{\includegraphics[width=0.45\linewidth]{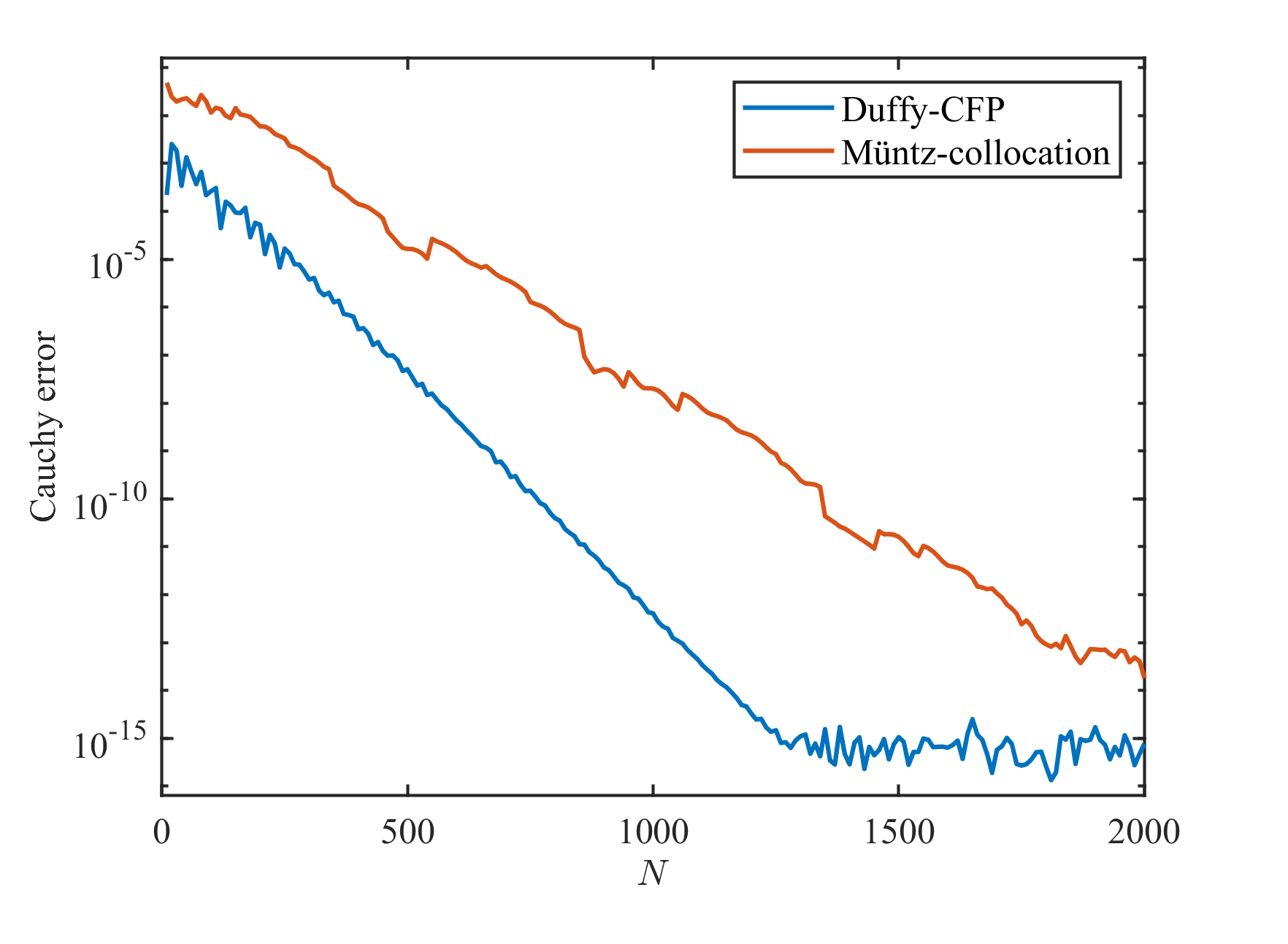}\label{fig:ex7_error}}
\hspace{0.1cm}
\subfloat[]{\includegraphics[width=0.45\linewidth]{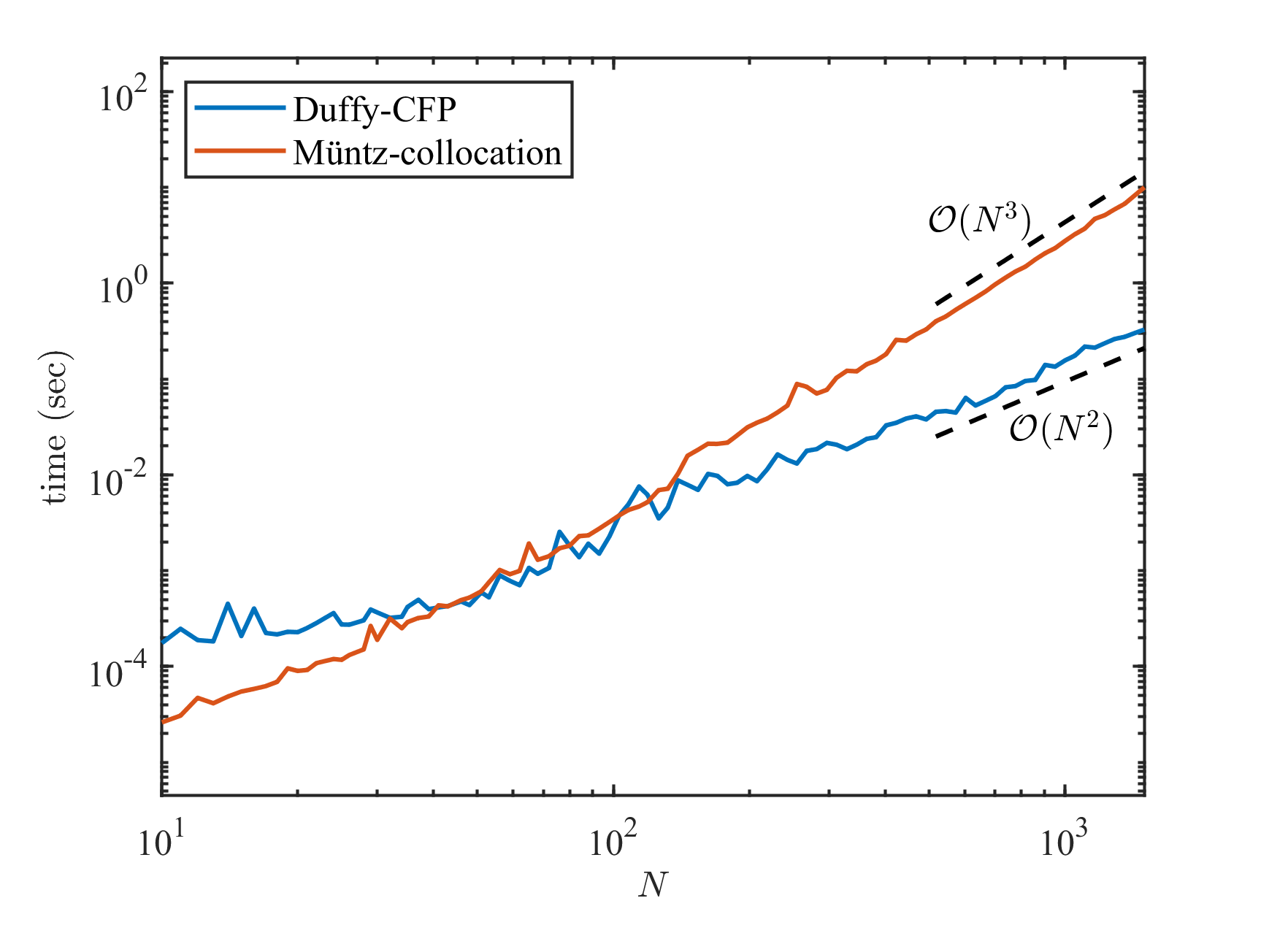}\label{fig:ex7_time}}
\caption{(A) Cauchy errors. (B) execution times.}
\label{fig:example_7}
\end{figure}

\section{Closing remarks}
This paper proposes a fast spectral framework for Volterra integral operators based on the Duffy transform and restriction operators. The method converts the approximation of kernels on triangular causal domains into spectral approximations on standard tensor-product domains. It then exploits recurrence relations to construct sparse matrix approximations of Volterra integral operators with smooth or weakly singular linear kernels and to accelerate the evaluation of nonlinear Volterra integrals. The resulting approximations lead to fast and well-conditioned spectral methods that outperform existing approaches in the tested examples, as demonstrated by the extensive numerical experiments.

Although smooth nonlinear and weakly singular linear Volterra integral operators are our focus, there is no fundamental obstacle to extending the method to singular nonlinear Volterra integral operators. We omit the discussion here, as few applications of such operators are known.

The restriction operator has many more applications than those considered here---it plays a crucial role in functional and delay differential equations, composition, inversion, and rootfinding of series expansions in classical orthogonal polynomials, among others. We defer these investigations to future works.


\appendix
\crefalias{section}{appendix}
\section{Identities for classical orthogonal polynomials}\label{sec:appendix}
The following identities are standard and can be found in classical references on orthogonal polynomials; see, for example,~\cite{sze,and}.
\begin{lemma}[Identities of Chebyshev and Jacobi polynomials]
    \label{lem:classical_polynomial_identities}
    For the Chebyshev polynomials of the first kind,
    \begin{subequations}
        \begin{align}
            T_{n+1}(x)
            &=2xT_n(x)-T_{n-1}(x),
            \qquad n\geqslant1,
            \label{eq:chebyshev_three_term}\\
            T_m(x)T_n(x)
            &=\frac{1}{2}
            \bigl(
                T_{m+n}(x)+T_{|m-n|}(x)
            \bigr),
            \qquad m,n\geqslant0,
            \label{eq:chebyshev_product}\\
            \int_{-1}^{1}
            T_m(x)T_n(x)\frac{\md x}{\sqrt{1-x^2}}
            &=\frac{\pi}{2-\delta_{m0}} \delta_{mn},
            \qquad m,n\geqslant0,
            \label{eq:chebyshev_orthogonality}\\
            \int_{-1}^{1}T_n(x)\,\md x
            &=
            \begin{cases}
                \displaystyle\frac{2}{1-n^2},
                & n \text{ is even},\\
                0,
                & n \text{ is odd},
            \end{cases}
            \qquad n\geqslant0,
            \label{eq:chebyshev_integral}\\
            \frac{T_i(x)}{\sqrt{1-x^2}}
            &=\frac{(-1)^i2^ii!}{(2i)!}
            \frac{\md^i}{\md x^i}
            (1-x^2)^{i-\frac12},
            \qquad i\geqslant0,
            \label{eq:chebyshev_rodrigues}\\
            \frac{\md^i}{\md x^i}T_j(x)
            &=2^{i-1}j(i-1)!
            C_{j-i}^{(i)}(x),
            \qquad 1\leqslant i\leqslant j.
            \label{eq:chebyshev_derivative}
        \end{align}
    \end{subequations}
For the Jacobi polynomials with $\alpha,\beta>-1$ and the generalized Jacobi polynomials with one parameter equal to $-1$,
    \begin{subequations}
        \begin{align}
            P_n^{(\beta,\alpha)}(-x)
            &=
            (-1)^nP_n^{(\alpha,\beta)}(x),\quad
            P_n^{(\beta,-1)}(x)
            =
            (-1)^nP_n^{(-1,\beta)}(-x),
            \label{eq:jacobi_parity}
            \\
            xP_n^{(\alpha,\beta)}(x)
            &=
            a_n^{(\alpha,\beta)}
            P_{n+1}^{(\alpha,\beta)}(x)
            +
            b_n^{(\alpha,\beta)}
            P_n^{(\alpha,\beta)}(x)
            +
            c_n^{(\alpha,\beta)}
            P_{n-1}^{(\alpha,\beta)}(x),
            \label{eq:jacobi_multiplication}\\
            P_0^{(-1,\beta)}(x)
            &=P_0^{(\beta,-1)}(x)
            =1,
            \label{eq:jacobi_minus_one_zero}\\
            P_n^{(-1,\beta)}(x)
            &=
            \frac{n+\beta}{2n}
            (x-1)P_{n-1}^{(1,\beta)}(x),
            \qquad n\geqslant1,
            \label{eq:jacobi_minus_one_left}\\
            P_n^{(-1,\beta)}(x)
            &=
            \frac{n+\beta}{2n+\beta}
            P_n^{(-1,\beta+1)}(x)
            +
            \frac{n-1}{2n+\beta}
            P_{n-1}^{(-1,\beta+1)}(x),
            \qquad n\geqslant1,
            \label{eq:jacobi_parameter_shift_right}\\
            \frac{1+x}{2}P_n^{(-1,\beta+1)}(x)
            &=
            \frac{n+1}{2n+\beta+1}
            P_{n+1}^{(-1,\beta)}(x)
            +
            \frac{n+\beta+1}{2n+\beta+1}
            P_n^{(-1,\beta)}(x),
            \quad n\geqslant0,
            \label{eq:jacobi_parameter_shift_left}\\
            \frac{\md}{\md x}P_n^{(0,\beta)}(x)
            &=
            \frac{n+\beta+1}{2}
            P_{n-1}^{(1,\beta+1)}(x),
            \qquad n\geqslant1,
            \label{eq:jacobi_derivative}\\
            \bigl(
                (1+x)\frac{\md}{\md x}+\beta
            \bigr)P_n^{(0,\beta)}(x)
            &=
            (n+\beta)
            P_n^{(1,\beta-1)}(x),
            \qquad n\geqslant0.
            \label{eq:jacobi_differential_shift}
        \end{align}
    \end{subequations}
where $a_n^{(\alpha,\beta)}$, $b_n^{(\alpha,\beta)}$, and $c_n^{(\alpha,\beta)}$ are the coefficients in the standard three-term recurrence \cref{eq:jacobi_multiplication}; see \cite[(4.5.1)]{sze} for their expressions.
\end{lemma}

\makeatletter
\immediate\write\@auxout{\string\citation{bst-control}}
\makeatother

\bibliography{references}

\end{document}